\documentclass[11pt]{article}
\usepackage{graphicx,mathrsfs,amsmath,braket,units,amssymb,amsthm,
amsmath,amsbsy,amsfonts,amssymb,graphics,epsfig,float,mathrsfs,amsthm,braket,units}
\usepackage[retainorgcmds]{IEEEtrantools}
\newcommand{\pd}[2]{\frac{\partial #1}{\partial #2}}
\newcommand{\pdd}[2]{\frac{\partial^2 #1}{\partial {#2}^2}}
\newcommand{\pddd}[3]{\frac{\partial^2 #1}{\partial {#2} \partial {#3}}}
\newcommand{\sd}[2]{\frac{\textrm{d} #1}{\textrm{d} #2}}
\newcommand{\sdd}[2]{\frac{\textrm{d}^2 #1}{\textrm{d} {#2}^2}}

\newcommand{\benE}[2]{\begin{IEEEeqnarray}{#1}\label{#2}}
\newcommand{\eenE}{\end{IEEEeqnarray}}
\newcommand{\be}[1]{\begin{equation}\label{#1}}
\newcommand{\ee}{\end{equation}}
\newcommand{\bes}[1]{\begin{equation}\label{#1} \begin{split}}
\newcommand{\ees}{\end{split} \end{equation}}
\newcommand{\ben}[1]{\begin{eqnarray}\label{#1}}
\newcommand{\een}{\end{eqnarray}}

\newcommand{\re}[1]{~(\ref{#1})}
\newcommand{\ret}[2]{~(\ref{#1},\ref{#2})}

\newcommand{\bma}{\begin{pmatrix}}
\newcommand{\ema}{\end{pmatrix}}

\newcommand{\mK}[3]{K_{#1,[#2,#3]}}

\newtheorem{theorem}{Theorem}
\newtheorem{proposition}{Proposition}
\theoremstyle{definition} \newtheorem{definition}{Definition}
\newtheorem{corollary}{Corollary}
\newtheorem{lemma}{Lemma}
\begin{document}

\title{Geometric conditions for the positive definiteness of the second variation in one-dimensional problems
}


\author{Thomas Lessinnes \&
        Alain Goriely\\
        {\it  Mathematical Institute, University of Oxford}           
}

\date{}

\maketitle

\begin{abstract}
Given a functional for a one-dimensional physical system, a classical problem is to minimize it by finding  stationary solutions and then checking the positive definiteness of the second variation. Establishing the positive definiteness is, in general, analytically untractable. However, we show here that a global geometric analysis of the  phase-plane trajectories associated with the stationary  solutions leads to generic  conditions for  minimality. These results provide a straightforward and direct proof of positive definiteness, or lack thereof, in many important cases. 
In particular, when applied to mechanical systems, the stability or instability  of entire classes of solutions can be obtained effortlessly from their geometry in phase-plane, as illustrated on a problem of a mass hanging from an elastic rod with intrinsic curvature.
\end{abstract}

\section{Introduction} 

A central problem in the theory of optimisation is to find a function $\theta(s)$ such that some functional $\mathscr E[\theta]$ is locally or globally minimal over a certain space of allowable  functions~\cite{ce83}. In physics, this question arises for instance when considering the state of a mechanical system described at each instant by a function $\theta$ of one (e.g. spatial) variable. Let $\mathscr E[\theta]$ be the potential energy of that state. If a state $\theta$ minimises $\mathscr E$, then this state is a stable equilibrium of the system. Indeed, by contradiction, starting at $\theta$ with no kinetic energy, the system will remain stationary since any motion would require an increase in both  potential and kinetic energies and hence violate the conservation of the total energy. 

When minimising a $C^2$ function of one variable, say $f(x)$, we typically require that two conditions are met by $f$: the first derivative of $f$ must vanish at a point,  in which case we say that the point is stationary, and the second derivative of $f$ must be positive. Points which realise both these conditions are minima of $f$. Similarly, the conditions under which a functional $\mathscr E[\theta]$ is stationary are well known: $\theta$ must satisfy the Euler-Lagrange equations associated to $\mathscr E$~\cite{ce83,gefo00}. The question of whether a stationary function $\theta$ minimises $\mathscr E$ locally is more difficult. In general, it is sufficient that the second variation of $\mathscr E$  at $\theta$ is strictly positive definite and it is necessary that it is  positive definite~\cite{gefo00}. However, for practical problems, general methods allowing to systematically check these conditions remain elusive. 

 A key issue is that the question of  positive definiteness depends on the boundary conditions. In the case of Dirichlet boundary conditions, the theory of conjugate points fully addresses the issue~\cite{gefo00}. The basic idea  is to reduce the problem by looking at the spectrum of a Sturm-Liouville operator $\mathcal S$ (cf. Section~\ref{sec-num}) associated with  the second variation~\cite{mo51,ma09}. The second variation is strictly positive if, and only if, all  eigenvalues of $\mathcal S$ are positive on the space of perturbations compatible with the boundary conditions. Manning~\cite{ma09} generalised this strategy for the Neumann problem and presented a numerical method to determine the positive definiteness of the second variation of $\mathscr E$ once an explicit expression of the equilibrium is known.

The aim of the present study is to obtain conditions for positive-definiteness of the second variation based on the geometry of the stationary solutions in phase space. Therefore, these conditions do not  require detailed knowledge of the stationary function but only  of their global properties. Here, we focus on functionals which are the sum of a quadratic term in $\theta'$ and a term $V(\theta)$ that only depends on $\theta$ and we look for minimisers of $\mathscr E$ among a class of functions satisfying given boundary conditions (either fixed or free).  

We show that the stability of a stationary function $\theta$ can be assessed in many cases by defining an index corresponding to the number of times the trajectory  ($\theta$, $\theta'$)  in phase space crosses either the horizontal axis for Dirichlet boundary conditions or the vertical lines corresponding to the extremal points of $V$ for Neumann boundary conditions. In these cases,  the stability of $\theta$ is directly established and the  the second variation of $\mathscr E$ does not need to be studied nor does the associated Sturm-Liouville problem. 

In this paper we first give  a concise statement of the problem and of the main results. Then, we define the  second variation of the functional $\mathscr E$ and summarize a numerical method to check its positive-definiteness.  A formal statement of the main results and the proofs are then given.  Finally, we illustrate these ideas with a study of the problem of determining the stability of the equilibria of a massless, planar, intrinsically coiled, elastic rod pinned to an anchor and used to suspend a massive body.

\section{Problem statement\label{sec-state}, definitions and summary of main results} 
We consider a system whose state is described by a function of one argument $\theta(s)$ where $s\in[a,b]$. Let $\mathscr E$ be the functional 
\ben{myfunc}
\mathscr E[\theta]=  \int_a^{b} \mathscr L\big( \theta(s),\theta'(s) \big)   \textrm{d} s,
\een
with
\be{FormL}
\mathscr L( \theta,\theta')= \frac{(\theta'-A)^2} 2 - V(\theta),
\ee
where $(\ )'=\sd{(\ )}{s}$, $V$ is a $C^2$ function and $A$ is a real constant. Furthermore, we require that $\sd{V}{\theta}$ and $\sdd{V}{\theta}$ do not vanish simultaneously.
 
We are interested in  functions $\theta$  that locally minimise the functional $\mathscr E$ on the space of functions $C^1([a,b])$ with the norm
\be{defnorm}
||x(s)|| = \sup_{a\leq s\leq b} |x(s)| + \sup_{a\leq s \leq b}|x'(s)|. 
\ee
 Two types of boundary conditions will be considered here: \textit{fixed boundaries} for which the value of $\theta$ is prescribed at the ends: $\theta(a)=T_a$ and $\theta(b)=T_b$ where $T_a$ and $T_b$ are real valued constants; and \textit{free boundaries} for which there is no restriction at the tips. 

To find a local minimiser, we consider admissible perturbations of $\theta$. A perturbation $\tau$ for our problem is said to be \textit{admissible} if $\tau \in C^1([a,b])\setminus\{0\}$ and the perturbed function $\theta+ \epsilon \tau$ satisfies the same boundary conditions as the function $\theta$ for all $\epsilon$. For problems with fixed boundaries, the set of admissible perturbation is
\be{}
\mathcal C^D([a,b])=\{ \tau \in C^1([a,b]):~\tau(a)=\tau(b)=0\}\setminus\{0\}.
\ee
 For free boundary conditions, all $C^1([a,b])$ perturbations have to be considered. However we first focus on the following space of admissible perturbations\footnote{The superscripts $D$ and $N$ used to denote the space of admissible perturbations $\mathcal C^D$ and $\mathcal C^N$ respectively refer to the Dirichlet and Neumann boundary conditions that naturally arise when minimising $\mathscr E$ with respectively fixed and free boundaries.} 
 \be{}
 \mathcal C^N([a,b])=\{ \tau \in C^1([a,b]):~\tau'(a)=\tau'(b)=0\}\setminus\{0\}. 
 \ee
 Then we show in Appendix~\ref{app-CnvsC1} that in all cases considered here, a function $\theta$ is minimal with respect to perturbations in $C^1([a,b])$ iff it is minimal with respect to perturbations in $\mathcal C^N$. 

A function $\theta$ is  \textit{locally minimal}  for the functional $\mathscr E$  if for all \textit{admissible perturbations} $\tau$, there exists a real number $M>0$ such that for all $\epsilon\in [-M, M]\setminus\{0\}$,
 \be{minbasics}
\mathscr E[\theta+ \epsilon \tau]> \mathscr E[\theta].
 \ee
Since $M$ can be chosen arbitrarily small, we expand the left side of the inequality\re{minbasics}: 
\be{taylorexpand}
\mathscr E[\theta+ \epsilon \tau]= \mathscr E[\theta] + \left .\sd {\mathscr E[\theta+ \epsilon \tau]}{\epsilon}\right |_{\epsilon=0}~ \epsilon +\left . \sdd{\mathscr E[\theta+ \epsilon \tau]}{\epsilon}\right |_{\epsilon=0} ~\frac{\epsilon^2}{2} + O(\epsilon^3). 
\ee
The \textit{first variation} $\delta \mathscr E_\theta: \mathcal C^X([a,b])\to \mathbb R$ (where $X$ stands for $D$ or $N$)  is defined for a given $\theta$ by
 \be{defvariations}
\delta \mathscr E_\theta [\tau]= \left .\sd {\mathscr E[\theta+ \epsilon \tau]}{\epsilon}\right |_{\epsilon=0}, 
 \ee
 Similarly, the \textit{second variation} $\delta^2\mathscr E_\theta:\mathcal C^X([a,b])\to \mathbb R$ is 
\be{defvariations2}
\delta^2 \mathscr E_\theta [\tau]= \left .\sdd {\mathscr E[\theta+ \epsilon \tau]}{\epsilon}\right |_{\epsilon=0}.
\ee
A necessary condition for $\theta$ to be minimal is \cite{gefo00}
\be{ELintro}
\forall \tau \in \mathcal C^X([a,b]):  ~\delta\mathscr E_\theta [\tau] = 0.
\ee
Functions $\theta$ which satisfy\re{ELintro} are called \textit{stationary} with respect to the functional $\mathscr E$. 

Once a stationary function $\theta$ is known, the problem is to determine if it is a minimum of the functional.
Since the first order terms in\re{taylorexpand} vanish on $\theta$, the inequality\re{minbasics} is  dominated by the second order term. If it is \textit{strongly positive}: 
\be{}
\exists k\in\mathbb R_0^+:~\forall \tau \in \mathcal C^X([a,b]): \delta^2 \mathscr E_\theta [\tau] \geq k ~||\tau||^2, 
\ee
then the stationary function is a local minimum~\cite{gefo00}. However, we show in Appendix~\ref{app-stronglypositive} that if the second variation is strongly positive with respect to the $\textrm{L}^2$ norm: 
\be{secvar}
\exists k\in\mathbb R_0^+:~\forall \tau \in \mathcal C^X([a,b]): \delta^2 \mathscr E_\theta [\tau] \geq k ~\int_a^b \tau^2(s)~\textrm{d}s, 
\ee
then, $\theta$ is minimal for the problem\ret{myfunc}{FormL}. 

If there exists $\tau$ such that the second variation is negative, then the stationary function is not a minimum. Finally, the case where $\delta^2 \mathscr E_\theta$ is positive on $C^X([a,b])$ but vanishes identically for some non-trivial $\tau$ requires the study of  higher order terms in\re{taylorexpand}, a case not considered here.

The vanishing of the first variation\re{ELintro} with either types of boundary condition implies (see e.g.~\cite{gefo00}) that stationary functions solve the Euler-Lagrange equations associated with the functional $\mathscr E$:
\ben{firstvariation}
\pd{\mathscr L}{\theta}-\sd{}{s}\pd{\mathscr L}{\theta '}=0.
\een
For free boundaries, Eq.\re{ELintro} also implies that the generalised moment $\pd{\mathscr L}{\theta'}$ associated with $\mathscr E$ vanishes at the tips:
\be{natBC}
\left .\pd{\mathscr L}{\theta'}\right|_{(\theta(a),\theta'(a))}=\left .\pd{\mathscr L}{\theta'}\right|_{(\theta(b),\theta'(b))}=0.
\ee
The conditions\re{natBC} are sometimes called \textit{natural boundary conditions} for the second order differential equation\re{firstvariation}. In our case, this leads to the Neumann boundary conditions: $\theta'(a)=\theta'(b)=A$.

For the particular form of $\mathscr L$, Eq.\re{firstvariation} takes the form
\be{thetapppot} 
\theta''(s) + \left .\sd{V}{\theta}\right|_{\theta(s)} = 0 ,\quad \textrm{with } \quad 
\left \{
\begin{split}
&\theta(a)=T_a \textrm{ and } \theta(b)=T_b~~ \textrm{fixed boundaries},\\
&\theta'(a)=A \textrm{ and } \theta'(b)=A ~~\textrm{free boundaries}.
\end {split} 
\right .
\ee

Note that stationary functions for the variational problem with fixed boundaries are described by a boundary value problem with Dirichlet boundary conditions while for the particular functional\ret{myfunc}{FormL}, the case of free boundaries leads to a BVP with Neumann boundary conditions.

If $\mathscr E$ is the energy of a physical system,  its stationary functions are called \textit{equilibria}. If an equilibrium is  a local minimum then the equilibrium is said to be \textit{stable}, otherwise  it is \textit{unstable}.

A typical analogy used in mechanics is to view
\begin{equation}\label{thetapppot1}
\theta''(s) + \left .\sd{V}{\theta}\right|_{\theta(s)} = 0
\end{equation}
as a dynamical system in phase plane $(\theta,\theta')$ where $s$ plays the role of time.  Any solution $\theta(s),\ s\in[s_1,s_2]$  of the differential equation\re{thetapppot1} defines  an oriented  curve (a \textit{trajectory} $\eta$) in phase plane through the mapping $\eta:[s_1,s_2]\to\mathbb{R}^2:s\to (\theta(s),\theta'(s))$.  The particular trajectories that correspond to solutions $\theta(s)$ of~(\ref{thetapppot}) will be denoted by $\gamma:[a,b]\to\mathbb{R}^2:s\to (\theta(s),\theta'(s))$. Since $\mathscr L$ does not depend explicitly on $s$, the associated Hamiltonian $H=\pd{\mathscr L}{\theta'} \theta'-\mathscr L $ is constant along trajectories and provides a first integral of\re{thetapppot}, the \textit{pseudo-energy}
\be{defE}
E=\frac{{\theta'}^2}{2} +V(\theta).
\ee
This analogy associates every solution of the boundary-value problem with a solution of an initial value problem of  a point mass  in a potential $V$.  An example of a trajectory in phase plane is  shown in Fig.~\ref{fig-V}   together with the motion of the point mass in the potential $V$.  The correspondence between the two problems provides a powerful tool to classify different solutions of a boundary-value problem and is known as a \textit{dynamical analogy} or \textit{Kirchhoff analogy} in the theory of one-dimensional elastic systems \cite{nigo99}. We further extend this analogy here to study the problem of stability by considering global geometric properties of the trajectories in phase plane.

\begin{figure}[h]
\centering\includegraphics[width=\linewidth]{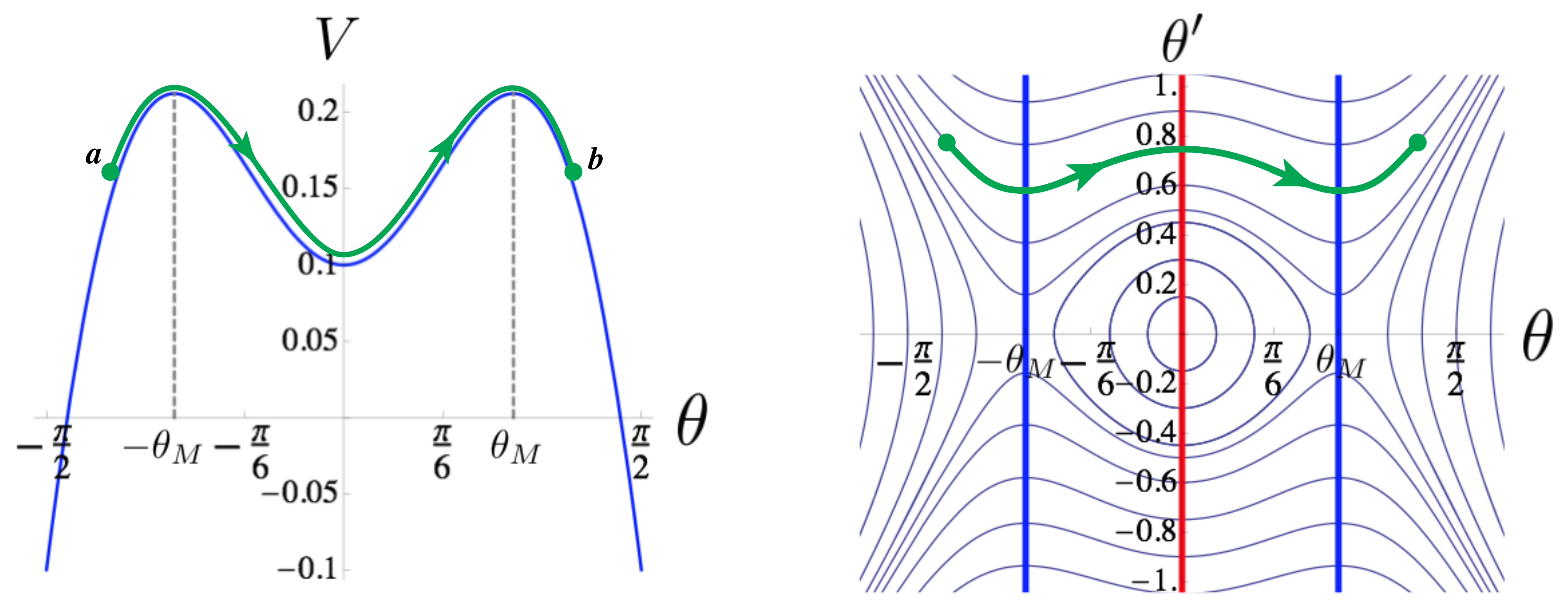}
\caption{Example of an arbitrary potential function $V$ (left), the associated phase space (right) with a few stationary trajectories, vertical max boundaries at $\pm \theta_M$ (blue online) and min boundary at $\theta=0$ (red online).} \label{fig-V}
\end{figure}

First,  for Dirichlet boundary conditions, we consider the number of times a phase plane trajectory $\eta$   crosses the horizontal axis $\Gamma_h=\{(\theta,\theta')\in\mathbb R^2\big|\theta'=0\}$ by defining the index 
\begin{equation}\label{definition}
I[\eta]=\#  \{ s\in[s_1,s_2]: \eta(s)\in\Gamma_h\}.
\end{equation}
We will establish that if $I[\gamma]=0$, then the second variation is positive definite. If $I[\gamma]\geq 2$, then the second variation is negative for some $\tau\in \mathcal C^D([a,b])$. The case $I[\gamma]=1$ requires the computation of a second global quantity.  Define  $L=L(\theta_1,\theta_2,E)$ as the flight time  from $\theta_1$ to $\theta_2$ along a solution of pseudo-energy $E$. If $\theta$ has pseudo-energy $E$ and $\theta(c)=\theta_1$, then $\theta(c+L)= \theta_2$, the second global quantity is then the derivative of the flight time with respect to $E$, that is
\begin{equation}
\alpha=\pd{L(\theta_1,\theta_2,E)}{E}.
\end{equation}
If $\alpha>0$, the second variation is positive definite whereas for $\alpha<0$, it is negative for some $\tau\in \mathcal C^D([a,b])$. The case $\alpha=0$ requires the computation of higher-order variations.

Second, for natural boundary conditions, we consider the number of time a trajectory crosses distinguished vertical boundaries. The extrema of $V(\theta)$ define vertical  lines which split the phase plane in different regions. The vertical lines that correspond to maxima of $V$ (resp. minima of $V$)  are called \textit{max-boundaries} (resp. \textit{min-boundaries}). For a given $V,$ the set of points on max-boundaries is $\Gamma_M=\{(\theta,\theta')\in\mathbb R^2\big|\, V_\theta(\theta)=0, \,V_{\theta\theta}(\theta)<0\}$  and the set of points on min-boundaries $\Gamma_m=\{(\theta,\theta')\in\mathbb R^2\big|\, V_\theta(\theta)=0, \,V_{\theta\theta}(\theta)>0\}$. 
For any trajectory $\eta$, the index 
\begin{equation}
\label{defJ}
J[\eta]=\# \left \{ s\in[s_1,s_2]: \eta(s)\in \Gamma_m\right\}-\# \left \{ s\in[s_1,s_2]: \eta(s)\in \Gamma_M \right\}
\end{equation}
 is the number of times the trajectory $\eta$ crosses min-boundaries minus the number of times it crosses max-boundaries. As an example, for the trajectory shown in Figure~\ref{fig-V} the index $J[\gamma]=1-2=-1$. We will establish that if $J[\gamma]<0$ the second variation is positive definite whereas for $J[\gamma]>0$, it is not positive definite. Hence, with no further computation, we conclude that the trajectory depicted in Figure~\ref{fig-V} is stable.
 The case $J[\gamma]=0$ requires the computation of a second global quantity, namely
 \begin{equation}
\beta=\xi(a)-\xi(b),\qquad \textrm{where\ \ \ \ }\xi(s)=\theta'(s)  \left.\sd{V}{\theta}\right|_{\theta(s)}.
\end{equation}
If $\beta\leq 0$ the second variation is not positive definite. The case $\beta>0$ remains inconclusive.
The different cases are summarised in Table 1.

\begin{table}[h!]
\begin{center}
\begin{tabular}{cc}
\begin{tabular}[t]{|l|c|c|}
\hline
\multicolumn{3}{|c|}{
{\large Dirichlet BC}}\\
\hline
\hline
\multicolumn{2}{|l|}{$I=0$ }& stable\\
\hline 
$I=1$ & $\alpha<0 $& unstable\\
&$\alpha>0 $& stable\\
\hline
\multicolumn{2}{|l|}{$I\geq 2$} & unstable\\
\hline
\end{tabular}
&
\begin{tabular}[t]{|c|c|c|}
\hline
\multicolumn{3}{|c|}{
{\large Natural BC}}\\
\hline
\hline
\multicolumn{2}{|l|}{$J<0$ }& stable\\
\hline 
$J=0$ & $\beta\leq 0 $& unstable\\
&$\beta>0 $& inconclusive\\
\hline
\multicolumn{2}{|l|}{$J>0$} & unstable\\
\hline
\end{tabular}
\end{tabular}
\end{center}
\caption{Summary of results. Stability and instability conditions for Dirichlet and Natural boundary conditions for a given stationary function $\theta$ with associated phase-plane trajectory $\gamma$. Here $I=I[\gamma]$ and $J=J[\gamma]$ as defined in the text.
}
\end{table}

\section{The second variation\label{sec-secvar}}

The second variation of $\mathscr E[\theta]$ defined by Equation~(\ref{defvariations2}) can be expanded as follows
\ben{}
\delta^2 \mathscr E_\theta [\tau]&=& \left . \left (\sdd{ ~~}{\epsilon}\mathscr E[\theta+\epsilon \tau]\right ) \right|_{\epsilon=0}\nonumber\\&=&
\int_a^{b} \left .\left(\sdd{ ~~}{\epsilon}  
\mathscr L(\theta + \epsilon \tau, \theta'+\epsilon \tau')\right) 
\right |_{\epsilon=0} \textrm{d} s\nonumber\\
&=& \int_a^{b} \Bigg\{ 
\left. \pdd {\mathscr L}{\theta} \right |_{(\theta(s),\theta'(s))}      \tau^2+
 \left.2 \pddd{\mathscr L}{\theta'}{\theta} \right |_{(\theta(s),\theta'(s))}    \tau \tau'\nonumber\\
 &&\qquad\qquad\qquad\qquad\qquad\qquad+
 \left.\pdd {\mathscr L}{\theta'}\right |_{(\theta(s),\theta'(s))}  (\tau')^2 \bigg\} ~\textrm{d}s\label{secvar1}\nonumber\\
&=&\int_a^{b} (\tau')^2 - \sdd {V}{\theta}\bigg|_{\theta(s)} \tau^2~~ \textrm{d} s ,\label{secvar2}
\een
where\re{secvar2} follows using the form of $\mathscr L$ defined in\re{FormL}.

Therefore, for a stationary function $\theta$, the sufficient condition\re{secvar} for $\theta$ to be minimal is that there exists a number $k>0$ such that 
\be{Secders}
\forall \tau\in \mathcal C^X([a,b]):\qquad\int_a^{b} (\tau')^2 -  \sdd {V}{\theta}\bigg|_{\theta(s)} \tau^2~ \textrm{d} s \geq k \int_a^b \tau^2(s)~\textrm{d}s .
\ee
Conversely, if
\be{Secderu}
\exists \tau\in \mathcal C^X([a,b]):\qquad \int_a^{b} (\tau')^2 -  \sdd {V}{\theta}\bigg|_{\theta(s)} \tau^2 \textrm{d} s <0.
\ee
then the stationary function $\theta$ is not a minimum. 
 
\begin{proposition}
If a solution $\theta(s)$ of\re{thetapppot} remains in a domain where $\sdd{V}{\theta}<0$, then this solution is minimal. Conversely, for natural boundary conditions: a solution that remains in a  domain where $\sdd{V}{\theta}>0$ is not minimal.
\label{cor-2}
\end{proposition}
\begin{proof}
The first statement follows from the fact that $\sdd{V}{\theta}<0$ implies that the condition\re{Secders} is satisfied with 
$
k=\inf_{s\in[a,b]} \left(- \left.\sdd{V}{\theta}\right|_{\theta(s)}\right).
$
 The solution is therefore minimal. 

The second statement can be established by choosing the perturbation $\tau =1$, which satisfies the natural boundary conditions and for which the integrand in\re{Secderu} is everywhere negative when $\sdd{V}{\theta}>0$. \qed
\end{proof}

\begin{proposition}
\label{cor-3}
In the case of natural boundary conditions, a constant solution, $\theta(s)=C\in  \mathbb R$,  of\re{thetapppot} 
is minimal if and only if $\left.V_{\theta\theta}\right|_{C}<0$. 
\end{proposition}
\begin{proof}
The results follows from a direct application of Proposition~\ref{cor-2}. Note that the case $\left.V_{\theta\theta}\right|_{C}=0$ is ruled out by our assumption that $V_\theta$ and $V_{\theta\theta}$ never vanish simultaneously. \qed
\end{proof}

\section{A numerical strategy\label{sec-num}}

In this section, we summarise the main steps of the method developed in~\cite{ma09} to establish the inequality\re{Secders} or\re{Secderu} for a given stationary function $\theta$. The key ideas are given without proof as they can be found in the original work. 

Integrating the first term in\re{Secders} by part when $\theta(s)$ is a solution of\re{thetapppot} leads to
\be{StaCond}
\exists k>0:~ \forall \tau \in \mathcal D^X([a,b]):~\int_a^{b} \tau \bigg( - \tau''- \sdd{V}{\theta} \tau \bigg)~\textrm{d}s \geq k \int_a^b \tau^2(s)~\textrm{d}s,
\ee
where $X$ stands for $D$ or $N$. The spaces of admissible $C^2$ perturbations
\ben{BCtauD}
&&\mathcal D^D([a,b])\equiv \{\tau \in C^2([a,b])\setminus\{0\}:~\tau(a)=0,~\tau(b)=0\}\\
&&\label{BCtauN}\mathcal D^N([a,b])\equiv \{\tau \in C^2([a,b])\setminus\{0\}:~\tau'(a)=0,~\tau'(b)=0\}%
\een
are dense in $\mathcal C^X([a,b])$ so that\re{StaCond} is equivalent to\re{Secders} which implies\re{minbasics}.
We use the standard inner product of functions in the space $\mathcal D^X([a,b])$,
 \be{bra}
\braket{x|y}=\int_a^{b} x(s) y(s) \textrm{d}s,
\ee
to express\re{StaCond} as 
\be{StaCondOp}
\exists k>0:~~ \forall \tau \in \mathcal D^X([a,b)]:~~\braket{ \tau | \mathcal S \tau} \geq k \braket{\tau|\tau},
\ee
where 
\begin{equation}
\mathcal S =-\sdd{}{s} +f(s), \qquad f(s)=- \sdd{V}{\theta}\big|_{\theta(s)},
\end{equation}
is a second order Sturm-Liouville linear differential operator. In particular, it is self-adjoint and its spectrum on $\mathcal D^X([a,b])$ is given by a discrete set of real eigenvalues. We conclude that\re{StaCondOp} is true if and only if the eigenvalues of $\mathcal S$ on $\mathcal D^X([a,b])$ are all strictly positive
\footnote{
Finding stationary functions for the functional $\mathscr E$ amounts to solving the boundary value problem\re{thetapppot}. We had noted that in the case of fixed boundaries, this BVP has Dirichlet boundary conditions while for free boundaries, the BVP has Neumann boundary conditions.  Something similar happens here. For fixed boundaries, the stability can be assessed by solving a Sturm-Liouville problem with Dirichlet boundary conditions while for free boundaries, the S.-L. problem has Neumann boundary conditions. 
}.

The strategy developed by Manning~\cite{ma09} is divided into two steps.  First, the eigenvalues of $\mathcal S$ are computed on an asymptotically small domain $\mathcal D^X([a,\sigma])$ with $\sigma\to a^+$. These eigenvalues are referred to as \textit{inborn eigenvalues}. Second, from the Sturm-Liouville theory, we know that the  eigenvalues of $\mathcal S$ on $\mathcal D^X([a,\sigma])$ depend  smoothly on $\sigma$~\cite{ka80,koze96} (see also Appendix~\ref{app-itworks}). Therefore, as
$\sigma$ increases up to $b$ the changes of sign of eigenvalues of $\mathcal S$ on $\mathcal D^X([a,\sigma])$ are monitored together with the direction of the change (from positive to negative or vice-versa). This process allows to count the total number of negative eigenvalues when $\sigma=b$ which in turn determines the positive-definiteness of $\mathcal S$.

The inborn eigenvalues of $\mathcal S$ are determined by noting that 
\begin{equation}
\mathcal S\mathop{\to}_{\sigma\to a} \mathcal S_0=- \sdd{}{s}+ f(a).
\end{equation}
The linear and homogeneous differential operator $\mathcal S_0$ has constant coefficients and its eigenvalues on $\mathcal D^X([a,\sigma])$ are
 \benE{lll} {inborneigen}
\lambda_k=f(a)+\frac{k^2\pi^2}{(\sigma-a)^2},\quad &\textrm{with } k\in \mathbb Z\setminus\{0\}~~ &\textrm{ if } X=D,\\
\lambda_k=f(a)+\frac{k^2\pi^2}{(\sigma-a)^2},\quad &\textrm{with } k\in \mathbb Z &\textrm{ if } X=N.
\eenE
For Dirichlet boundary conditions ($X=D$), and for $\sigma$ sufficiently close to $a$, $\lambda_k>0$ for all $k$. For natural boundary conditions ($X=N$),  and for $\sigma$ sufficiently close to $a$, $ \lambda_k>0 $ for all $k\neq0$. When $k=0$, $\lambda_0=f(a)$. Hence, we have: 
\begin{proposition}
\label{prop-inborn}
For natural boundary conditions, there is either one negative inborn eigenvalue if $f(a) <0$ or none if $f(a)>0$. For Dirichlet boundary conditions there are no negative inborn eigenvalues. 
\end{proposition}

As $\sigma$ increases up to $b$, the eigenvalues change continuously as functions of $\sigma$. 
\begin{definition} \label{def-conj}
A  value $\sigma_c>a$ such that
\be{zeroeig}
\mathcal S\tau=0, 
\ee
has a solution on $\mathcal D^X([a,\sigma_c])$ is a \textit{conjugate point to $a$}.
\end{definition}
This definition extends the notion of conjugate points developed for problems with Dirichlet boundary conditions (see for instance~\cite{gefo00} for an introduction; equivalence is discussed in~\cite{ma09}). Conjugate points are  particularly important since at each crossing in $\sigma$, one and only one eigenvalue of $\mathcal S$ on $\mathcal D^X([a,\sigma])$ changes sign (it vanishes at $\sigma_c$). See Appendix~\ref{app-itworks} for details.  

\begin{definition}\label{def-index}
The \textit{Index}$_{[\sigma_1,\sigma_2]^X}\, [\theta]$ of a solution $\theta$ on an interval $[\sigma_1,\sigma_2]$ is defined as the number of negative eigenvalues of the operator $\mathcal S$ on $\mathcal D^X([\sigma_1,\sigma_2])$.
\end{definition}
The number of sign changes tracked by the index is critical for the problem of stability. 
\begin{proposition}\cite{ma09}
\label{prop-Index}
If $b$ is not a conjugate point to $a$, and 
\be{Index}
\textrm{Index}_{[a,b]^X}\, [\theta]=0,
\ee
then $\theta$ is locally minimal for the functional\ret{myfunc}{FormL} with boundary conditions $X$.
\end{proposition}

The condition that $b$ is not conjugate to $a$ is necessary since, otherwise, there exists an eigenfunction on which the second variation of $\mathscr E$ vanishes and local minimality cannot be guaranteed.

The Index can be computed according to the following method. In the limit $\sigma\to a$, the index is given by the number of negative inborn eigenvalues. Then, as $\sigma$ increases, at each  conjugate point, an eigenvalue changes sign (the one that vanishes at the conjugate point). If a positive eigenvalue becomes negative, the index increases by 1. If \textit{vice versa} a negative eigenvalue becomes positive, it decreases by 1. In this computation of the Index, it is crucial to determine \textit{all}  conjugate points. The conjugate points can be obtained by computing the solution of an auxiliary problem.  First, we consider the case of Dirichlet boundary conditions.
\begin{proposition} \label{prop-DirConj}
Let $h_1$ be the solution of the initial value problem 
\be{IVPDM}
\mathcal S h_1=-h_1''+f\, h_1=0;\qquad h_1(a)=0;\quad h_1'(a)=1.
\ee
  Then, the conjugate points to $a$ for the associated Dirichlet problem are the roots of $h_1$. 
  \end{proposition}
\begin{proof}
The existence of a solution $h_1$ for this initial value problem on a closed interval is guaranteed by  the fact that $S h_1=0$ is a regular linear equation with continuous coefficients \cite[p.110]{go01}.  Assume that there exists a point $\sigma_c$ conjugate to $a$. According to Def.~\ref{def-conj},  there exists a function $\tau\in \mathcal D^D([a,\sigma_c])$ such that $\mathcal S\tau=0$. This function is such that $\tau(a)=0$ and $\tau'(a)\neq0$ (by contradiction, otherwise $\tau$ would vanish identically). Hence, the function $h_1={\tau}/{\tau'(a)}$ solves the the IVP\re{IVPDM}. In particular, it vanishes whenever $\tau$ vanishes.
 Conversely, let $\sigma_c$ be a root of $h_1$. Then the function $\tau(s)=h_1(s),\ s\in[a,\sigma_c]$, is such that $S\tau=0$ on $\mathcal D^D([a,\sigma_c])$, that is,  $\sigma_c$ is conjugate to $a$, according to Def.~\ref{def-conj}. \qed
  \end{proof}

Second, we give the
 analogous result for natural boundary conditions (given here without proof, see~\cite{ma09}).
\begin{proposition}\label{prop-NeuConj}
Let $h_2$ be the  solution of the initial value problem 
\be{IVPN}
\mathcal Sh_2=-h_2''+f\, h_2=0; \qquad h_2(a)=1;\quad h_2'(a)=0.
\ee
Then, the conjugate points to $a$ for the associated Neumann problem  are the roots of $h_2'$.
\end{proposition}

Depending on the boundary condition, the solution of the initial value problem\re{IVPDM} or\re{IVPN} reduces the problem of finding the conjugate points to finding   the roots of the IVP's solution. Numerically, this can be done by monitoring both the sign of the solution $h_1$ or $h_2$ as $\sigma$ increases as well as the sign of the corresponding eigenvalues~\cite{ma09}.

\section{Dirichlet boundary conditions}\label{sec-dir}

 The case of Dirichlet boundary conditions is easier to solve  because Index$_{[a,\sigma]^D}$ always increases as $\sigma$ crosses a conjugate point. This property is a consequence of the fact that the eigenvalues of the corresponding Sturm-Liouville problem decrease monotonically with the size of the domain as shown by Dauge and Helffer~\cite{dahe93} for the Sturm-Liouville problem 
 \begin{equation}
-(p y')' + q\,  y = \lambda \, w\, y,\quad y(a)=0,\ y(\sigma)=0
\end{equation}
where $\sigma>a$ and the function $p\geq k$ with $k$ a strictly positive number. In particular, they showed that the dependence of an eigenvalue on $\sigma$ is given by the equation
\be{DHD}
\sd{\lambda}{\sigma}= - p(\sigma)\, [u'(\sigma)]^2,
\ee
where $u$ is the $L^2$ normed eigenfunction associated with $\lambda$ on $\mathcal D^D([a,\sigma])$.  Since $p$ is strictly positive, $\sd{\lambda}{\sigma}$ is negative and all eigenvalues of such Sturm-Liouville problems with Dirichlet boundary conditions decrease as $\sigma$ increases.

We recall from Proposition~\ref{prop-inborn} that for Dirichlet boundary conditions, there are no negative inborn eigenvalues. 
 Since all eigenvalues decrease with increasing $\sigma$, at a conjugate point, a positive eigenvalue becomes negative. Once it becomes negative, it keeps decreasing as $\sigma$ increases and can never become positive again. This simple fact leads to another well known result~\cite{gefo00}: 
 \begin{proposition} \label{prop-DirConjEqMin}
 For Dirichlet boundary conditions, if there exists at least one conjugate point to $a$ in $(a,b)$, then $\textrm{Index}\, [\theta]>0$, and $\theta$ is not minimal. Conversely, if there are no conjugate points in $(a,b]$, then $\theta$ is minimal. 
 \end{proposition}

The open question is therefore to establish whether there is a conjugate point in $(a,b]$. We first consider the two simplest cases for the functional $\mathscr E$ defined in\ret{myfunc}{FormL}. We recall that the index $I[\gamma]$ is defined in\re{definition} as the number of times $\theta'(s)$ vanishes on $s\in[a,b]$. 
\begin{theorem} \label{theo-IndI}
Let $\theta$ be a stationary function of $\mathscr E$ such that $\theta'$ does not vanish uniformly on the interval $[a,b]$. Let $\gamma:s\in[a,b]\to(\theta(s),\theta'(s))$ be its associated phase plane trajectory. If $I[\gamma]\geq2$, then $\theta$ is not locally minimal for the functional $\mathcal E$ on $\mathcal D^D([a,b])$. If $I[\gamma]=0$, then $\theta$ is locally minimal for the functional $\mathcal E$ on $\mathcal D^D([a,b])$.
\end{theorem}
\begin{proof}

This result is a consequence of the Sturm separation theorem  which states that (\textit{e.g.}~\cite[Theorem 5.41]{kepe10}): ``If $x$ and $y$ are linearly independent solutions on an interval $I$ of the second order self-adjoint differential equation $L x=0$, then their zeros separate each others in $I$. By this we mean that $x$ and $y$  have no common zeros and between any two consecutive zeros of one of these solutions, there is exactly one zero of the other solution.'' 

In the previous section we defined $h_1$ as the unique solution of the initial value problem\re{IVPDM}. In particular, Sturm-Liouville operators are self-adjoint, therefore, $h_1$ solves the self-adjoint differential equation $\mathcal Sh_1=0$.  We also know that since $\theta$ is stationary, the function $\theta'$ is such that $\mathcal S \theta'=0$. Therefore, we have two solutions to the self-adjoint equation $\mathcal S h=0$.

Two cases must be distinguished. First, if $\theta'$ and $h_1$ are linearly dependent, they have identical roots and $I[\gamma]=0$ never occurs since $h_1(a)=0$ implies $\theta'(a)=0$ so that $\gamma(a)\in\Gamma_h$ and therefore $I[\gamma]\geq1$. If $I[\gamma]\geq2$, then $\theta'$ vanishes in $(a,b]$ and so does $h_1$. The point at which this happens is conjugate to $a$ by application of Proposition~\ref{prop-DirConj}. Then, Proposition~\ref{prop-DirConjEqMin} implies that $\theta'$ is not minimal. 

Second, consider the case where $\theta'$ and $h_1$ are linearly independent. If $I[\gamma]\geq2$, then $\theta'$ has at least two consecutive zeros in $[a,b]$ and the Sturm separation theorem implies that $h_1$ must vanish exactly once between these two zeros. The point at which $h_1$ vanishes is conjugate to $a$ by application of Proposition~\ref{prop-DirConj} and the solution $\theta$ is therefore not minimal (cf. Proposition~\ref{prop-DirConjEqMin}). Finally if $I[\gamma]=0$, then $h_1$ cannot vanish on $(a,b]$. By contradiction, if $h_1$ vanishes on $(a,b]$, we define $c$ to be the smallest root of $h_1$ on $[a,b]$. 
 Then $a$ and $c$ are consecutive zeros of $h_1$ and the Sturm separation theorem implies that $\theta'$ must vanish on $(a,c)$, a contradiction to $I[\gamma]=0$. So, if $I[\gamma]=0$, $h_1$ does not vanish on $(a,b]$, there are no conjugate points to $a$ (due to Proposition~\ref{prop-DirConj}) and $\theta$ is minimal as a consequence of Proposition~\ref{prop-DirConjEqMin}.\qed
\end{proof}

Next, we consider the case $I[\gamma]=1$.  We first define the length of an \textit{arc}, that is an oriented segment of  a trajectory  in  phase plane.
Consider a trajectory $\gamma: s\in A\subset \mathbb R\to (\theta(s),\theta'(s))$ in the phase plane of Eq.\re{thetapppot1} where $A$ is a closed interval in $\mathbb R$.  
We assume  that $I[\gamma]=1$, \textit{i.e.} there exists a unique $c\in A$ such that $\theta'(c)=0$ at which $\theta_c=\theta(c)$. 

Let $E$ be the pseudo-energy of $\gamma$. The value of $\theta_c$ depends on $E$ through $E=V(\theta_c)$. Choose two independent constants $T_0,P\in\mathbb R$ with the following properties: \textit{(i)} $\textrm{Sign}[\theta_c-T_0] P>0$; \textit{(ii)}    $\exists\ s_1\in A$ such that $T_0=\theta(s_1)$; \textit{(iii) } $\exists\ s_2\in A$ such that $P=\theta'(s_2)$. \\
Let $\eta$ be the arc  of $\gamma$ connecting the points $\big(T_0, P_0(T_0,E,P)\big)$ and $\big(T(E,P),P\big)$ where $P_0(T_0,E,P)=\textrm{Sign}[P] \sqrt{2 (E-V(T_0))}$ and where $T(E,P)$ is the root of $V(T)= E-\frac{P^2}{2}$,
that is closest to $\theta_c$ while respecting $\textrm{Sign}[T-T_0] P>0$.  we note that the arc $\eta=\eta(T_0,E,P)$ is fully specified through this construction by the three real numbers $(T_0,E,P)$ (see Fig.~\ref{fig-3}).
\begin{figure}
\centering\includegraphics[width=0.7\linewidth]{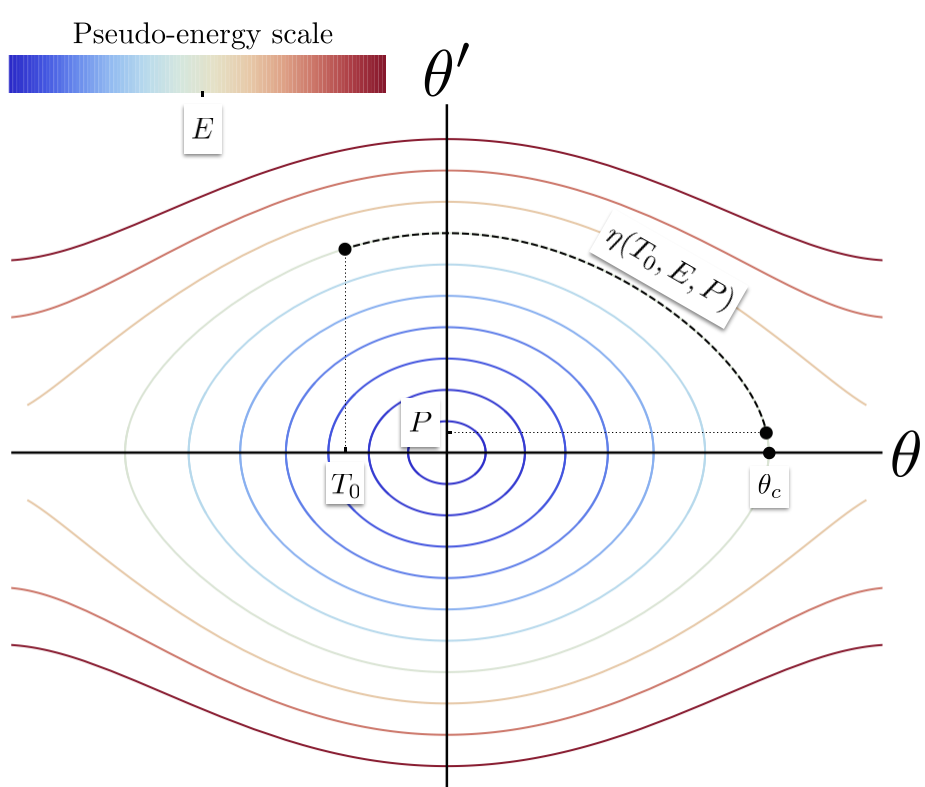}
\caption{In the phase plane an arc $\eta$ (dashed black) such that $I[\eta]=0$ is completely specified by prescribing its pseudo-energy $E$, its initial abscissa $T_0$, and its final ordinate $P$.} \label{fig-3}
\end{figure}

\begin{definition}
The \textit{length} $L(T_0,E,P)$ of the arc $\eta(T_0,E,P)$ is 
\be{def-Length}
L(T_0,E,P) = \frac {\textrm{Sign}[P]} {\sqrt{2}} \int_{T_0}^{T(E,P)} \frac{\textrm{d}\theta}{\sqrt{E- V(\theta)}}.
\ee
\end{definition}
The length $L$ is the size of the domain required for a solution of  energy $E$ to go from $T_0$ to $T$ when $P>0$ (or from $T$ to $T_0$ if $P<0$) without changing direction.  The variation of the length with respect to the pseudo-energy is of particular importance for the rest of this paper. It is given explictly by
\be{LE}
\pd{L}{E}(T_0,E,P)= \frac {1} {P~V_{\theta}(T(E,P)) } - \frac{\textrm{Sign}[P]}{2\sqrt{2}} \int_{T_0}^{T(E,P)} \frac{\textrm{d}\theta}{\left (E - V(\theta)\right)^{3/2}}.
\ee

The length $L(T_0,E,P)$ is only defined for arcs that have no intersection with the $\theta$-axis, but it can be used to define the length of an arc with one intersection. The length of a trajectory $\gamma$ with $I[\gamma]=1$ of pseudo-energy $E$ connecting the points $(T_1, P_1)$ to $\left(T_2,  P_2\right)$ is
\be{}
L[\gamma] = \lim_{P\to 0^{\textrm{Sign}[P_1]}} L(T_1,E,P) + \lim_{P\to 0^{\textrm{Sign}[P_2]}} L(T_2,E,P).
\ee
We can now state the general result for trajectories with a single intersection.

\begin{theorem} \label{prop-I1}
Let $\theta$ be a stationary function of $\mathscr E$ and let $\gamma$ be  its associated phase plane trajectory. 

If $I[\gamma]=1$ and $\pd{L}{E} >0$, then $\theta$ is locally minimal for the functional $\mathcal E$ on $\mathcal D^D([a,b])$.

If $I[\gamma]=1$ and $\pd{L}{E} < 0$, then $\theta$ is not locally minimal for the functional $\mathcal E$ on $\mathcal D^D([a,b])$.
\end{theorem}
\begin{proof}
 $I[\gamma]=1$ implies that there is a unique $c\in[a,b]$ such that $\theta'(c)=0$; let $\theta_c=\theta(c)$. 
Define the function 
\be{defmu}
\mu (s;d)= \theta'(s) ~\int_d^s\frac{\textrm{d}t}{\big(\theta'(t)\big)^2},\qquad d\in[a,b]\setminus\{c\}.
\ee
The domain of $\mu$ with respect to $s$ is defined as $[a,c)$ if $d<c$ and $(c,b]$ if $d>c$. We first  establish the following properties of the function $\mu$: 
\begin{enumerate} 
\item[(P1)] 
\be{}
\lim_{s\to c} \mu(s;d) = \frac{1}{\left .\sd{V}{\theta}\right |_{\theta_c} }.
\ee
This result is obtained by direct computation
\ben{}
\lim_{s\to c} \mu(s;d) & =& \lim_{s\to c} \frac{\theta'(s)}{\frac{1}{\int_d^s\frac{\textrm{d}t}{\theta'^2(t)}}} = \lim_{s\to c} \frac{\theta''(s)}{-\frac{1}{\left(\int_d^s\frac{\textrm{d}t}{\theta'^2(t)} \right)^2} \frac{1}{\theta'^2(s)}}  = \left . \sd{V}{\theta}\right|_{\theta_c} \lim_{s\to c} (\mu(s;d))^2.\nonumber
\een
 As a consequence, this limit does not depend on  $d$. 
 \item[(P2)] For any $s$ such that  $c \notin [\textrm{min}(s,d),\textrm{max}(s,d)]$, the derivative $\mu'(s;d)$ -- where $'$ denotes derivation by $s$ -- depends continuously and monotonically on $d$: it is strictly monotonically increasing with $d$ if $\left . \sd{V}{\theta}\right|_{\theta(s)}>0$ and strictly monotonically decreasing if $\left . \sd{V}{\theta}\right|_{\theta(s)}<0$. \\
This property follows from the explicit computation of $\mu'(s;d)$:
 \ben{proofP2}
\hspace{-.4cm}\mu'(s;d) = \frac{1}{\theta'(s)} + \theta''(s) \int_d^s \frac{\textrm{d}t}{\big(\theta'(t)\big)^2} = \frac{1}{\theta'(s)} - \left . \sd{V}{\theta}\right|_{\theta(s)} \int_d^s \frac{\textrm{d}t}{\big(\theta'(t)\big)^2}.
\een
\item[(P3)] If the arc connecting $(\theta(s), \theta'(s))$ to $(\theta(d),\theta'(d))$ does not intersect the $\theta$-axis, then
\be{p3}
\mu'(s;d) =\textrm{Sign}[\theta'(s)] \textrm{ Sign}[\theta(s)-\theta(d)] \left . \sd{V}{\theta}\right|_{\theta(s)}~ \pd{L}{E}(\theta(d),E,\theta'(s)).
\ee
To obtain this result, we rewrite the last expression in\re{proofP2} as
\be{}
\mu'(s;d)=\frac{1}{\theta'(s)} - \left . \sd{V}{\theta}\right|_{\theta(s)} \textrm{ Sign}[\theta'(s)] \int_{\theta(d)}^{\theta(s)} \frac{\textrm{d}\theta}{2\sqrt{2} \left (E-V(\theta)\right)^{3/2}},
\ee 
and use the definition\re{LE} to re-express the integral in terms of the derivative of the length with respect to the pseudo-energy.

\item[(P4)] \be{p4}
\lim_{s\to c^+ } \mu'(s,s)=  \textrm{Sign }[\theta'(b)]~ \infty.
\ee
\end{enumerate}
This result is obtained by taking the limit in\re{proofP2}. That is, we have \\ $\lim_{d\to s} \mu'(s;d) = \frac 1 {\theta'(s)}$. \\

Theorem 2 is based on Propositions~\ref{prop-DirConj} and~\ref{prop-DirConjEqMin}. We define $h_1$ to be the unique solution of the initial value problem $\mathcal S h_1=0$ 
with initial values $h_1(a)=0$ and $h_1'(a)=1$. Then, if $h_1$ has no root in $(a,b]$, then $\theta$ is minimal and if $h_1$ has a root in $(a,b)$, then $\theta$ is not minimal.  To prove the theorem, we prove separately the two statements: (A) $h_1$ has no root in $(a,c]$; (B) $h_1$ has a root in $(c,b)$ if and only if 
\be{limineq}
\theta'(a) \lim_{s\to c^-} \mu'(s;a) < \theta'(a) \lim_{s\to c^+} \mu'(s;b). 
\ee
This last condition\re{limineq} is then shown to be equivalent to 
\be{dLdEineq}
\pd{L}{E} < 0.
\ee

\noindent (A) First, we note that $\mathcal S\mu(s;a)=0$. This result follows from direct substitution and using the fact  that $\theta$ solves\re{thetapppot}. Noting that $\mu(a;a) = 0$ and $\mu'(a;a)=1/{\theta'(a)}$, we have
\be{h1L}
\forall s\in [a,c): h_1(s)= \theta'(a) \mu(s;a).
\ee
We observe that $\mu$ does not vanish on $(a,c]$.  Indeed, the function $\theta'$ does not vanish on $(a,c)$ since $c$ is its only root on $[a,b]$ and the integrand of the second factor in\re{defmu} is strictly positive. Finally, (P1) implies that $\mu$ does not vanish at $c$.\\

\noindent (B) For all  $d\in[a,b]\setminus \{c\}$, $\mu(s;d)$ and $\theta'$ are linearly independent solutions of $\mathcal S \tau=0$ since $\theta'(d)\neq0$ while $\mu(d;d)=0$. First we prove that if $h_1$ has a root in $(c,b)$, then\re{limineq} holds. Assume that there exists $d\in (c,b): h_1(d)=0$, then there exist  constants $C_1$ and $C_2$ such that $h_1(s) = C_1 \, \mu(s;d) + C_2\,  \theta'(s)$.
In this case, $h_1(d) = C_2\, \theta'(d)= 0 \Rightarrow C_2=0$. Next,\re{h1L} and (P1) imply that $C_1= \theta'(a)$. Hence, we have
\be{h1R0}
\exists d\in (c,b]:\, h_1(d)=0\qquad  \Rightarrow \qquad \forall s \in(c,b]: \, h_1(s)= \theta'(a)\,  \mu(s;d). 
\ee
Since $h_1 \in C^2([a,b])$, we have $\lim_{s\to c^-} h_1(s)= \lim_{s\to c^+} h_1(s) \Rightarrow \theta'(a) \mu'(c;a) = \theta'(a) \mu'(c;d)$. Finally, $\textrm{Sign}[\theta'(a)]= \textrm{Sign}\left [ \left.\sd{V}{\theta}\right|_{\theta_c} \right]$ and (P2) imply that\\ $\theta'(a) \mu'(c;d)$ is a strictly monotonically increasing function of $d$. Since $b>d$, the inequality\re{limineq} holds.

 Next, we prove that if\re{limineq} holds, then $h_1$ has a root in $(c,b)$. Multiplying\re{p4} by $\theta'(a)$ leads to $\lim_{s\to c^+ } \theta'(a)  \mu'(s;s) = -\infty$. Since according to (P2), $\theta'(a) \mu'(c;d)$ is continuous in $d$, it satisfies the assumption of the intermediate value theorem. Therefore, $\exists d^\star\in (c,b): \mu'(c;d^\star) = \mu'(c;a)$. For that particular value, the function 
\be{}
h_1(s)= 
\left \{
\begin{split}
\theta'(a) \mu(s; a) \quad &\textrm{if } s\in[a,c], \\
\theta'(a) \mu(s; d^\star) \quad &\textrm{if } s\in(c,b] ,
\end{split}
\right .
\ee
is $C^1([a,b])$ by construction and solves the IVP\re{IVPDM}. Accordingly, $h_1(d^\star)=\theta'(a) \mu(d^\star; d^\star)=0$ has a root at $d^\star \in(c,b)$. \\

Finally, to show the equivalence between\re{limineq} and $\pd{L}{E} < 0$, we substitute\re{p3} in\re{limineq} to obtain
\ben{}
&&\theta'(a) \left . \sd{V}{\theta}\right |_{\theta_c} ~\textrm{Sign}[\theta(c)-\theta(a)] \textrm{ Sign}[\theta'(c^-)]~~ \pd{L}{E}(\theta(a),E,0) \nonumber \\
&&\qquad <
\theta'(a) \left . \sd{V}{\theta}\right |_{\theta_c} \textrm{ Sign}[\theta(c)-\theta(b)] \textrm{ Sign}[\theta'(c^+)]~~ \pd{L}{E}(\theta(b),E,0).  \label{buildineq}
\een
In\re{buildineq}, the factor $\theta'(a) \left . \sd{V}{\theta}\right |_{\theta_c}$ is positive. Furthermore, we have $\textrm{Sign}[\theta(c)-\theta(a)] \textrm{ Sign}[\theta'(c^-)]=1$ and $\textrm{Sign}[\theta(c)-\theta(a)] \textrm{ Sign}[\theta'(c^+)]=-1$. Hence, this inequality simplifies to
\be{}
\pd{L}{E}(\theta(a),E,0) +\pd{L}{E}(\theta(b),E,0) <0.
\ee
\qed
\end{proof}

\section{Neumann boundary conditions\label{sec-neu}}

In Section~\ref{sec-num} we related the positive definiteness  of the second variation to the existence of negative eigenvalues of the second order differential operator $\mathcal S$.  For Dirichlet boundary conditions, this relation was equivalent to finding the roots of the solution of the IVP  $\mathcal S h_1=0$ with $h_1(a) =0$ and $h_1'(a)=1$. In Section~\ref{sec-dir}, we exploited the fact that $\theta'$ is also a solution $\mathcal S h_1=0$  to obtain sufficient conditions  for both the existence of negative eigenvalues (given by $I[\gamma]\geq 2$) and for their non-existence (given by $I[\gamma]=0$). Finally, we noted that the number of roots of $\theta'$ can be obtained in the phase plane by counting the number of times the associated trajectory crosses the \textit{horizontal} $\theta$-axis. 

In the case of natural boundary conditions, Proposition~\ref{prop-NeuConj} showed that the existence of negative eigenvalues  depends on the roots of $h_2'$ where $h_2$ is the unique solution to  the IVP $\mathcal S h_2=0$ with $h_2(a)=1$ and $h_2'(a)=0$. We now show that   the roots of $\theta''$ can be used to obtain sufficient conditions for both stability and instability. Since the roots of $\theta''$ are also roots of $\left. \sd{V}{\theta}\right|_{\theta}$,  stability or instability can  be obtained by  counting the number of times the associated trajectory crosses  the \textit{vertical} boundaries corresponding to extrema of the potential $V$. 

More precisely, recall from Section~\ref{sec-state}
that $\Gamma_m$ and $\Gamma_M$ are  min- and max-boundaries: the vertical lines corresponding to min and max of $V$. Then,  the index $J$ of an arc $\eta$  is defined by\re{defJ} as the number of times $\eta$ has intersects with  $\Gamma_m$ minus the number of times it intersects with $\Gamma_M$. This index is also the number of times $\eta$ crosses minimal points of $V$ minus the number of times it crosses maximal points of $V$. Recalling the definition of the
functional space $\mathcal D^N([a,b])= \{\tau\in C^2([a,b])\setminus\{0\}: \tau'(a)=\tau'(b)=0\}$, we have 
\begin{theorem}
\label{theo-IndJ}
Let $\theta:[a,b]\to \mathbb R$ be stationary for the functional $\mathscr E$  such that $\left. \sd{V}{\theta}\right|_{\theta(a)}\neq 0\neq \left. \sd{V}{\theta}\right|_{\theta(b)}$ and $\nexists s\in[a,b]: \theta'(s)=0\textrm{ and }\left . \sd{V}{\theta}\right|_{\theta(s)}=0$. Let $\gamma:s\in[a,b]\to (\theta(s),\theta'(s))\in\mathbb R^2$ be the trajectory associated to $\theta$ in the phase plane. 
Then,

 if $J[\gamma]>0$,   $\theta$ is not locally minimal for the functional $\mathcal E$ on $\mathcal D^N([a,b])$;

if $J[\gamma]<0$,  $\theta$ is locally minimal for the functional $\mathcal E$ on $\mathcal D^N([a,b])$.
\end{theorem}

We first  establish several intermediary results. Unlike for  Dirichlet boundary conditions, the eigenvalues of a Sturm-Liouville problem with Neumann boundary conditions do not decrease monotonically with the size of the domain.  Dauge-Helffer's theorem~\cite{dahe93} states that for the Sturm-Liouville problem $-(p y')' + q\,  y = \lambda \, y$ with boundary conditions $y'(a)=0$ and $y'(\sigma)=0$ where $\sigma>a$ and the function $p\geq k$ with $k$ a strictly positive number, we have  
\be{DHN}
\sd{\lambda}{\sigma}= (u(\sigma))^2~ (q(\sigma) - \lambda(\sigma)),
\ee 
where $u$ is a $L^2$ normed eigenfunction associated with $\lambda$ on $\mathcal D^N([a,\sigma])$. Therefore, for $\sigma$ such that $\lambda(\sigma)=0$, the sign of $\sd{\lambda}{\sigma}$ is  given by the sign of $q(\sigma)$.  This result is important to establish the following proposition:
\begin{proposition}
\label{prop-ourIndex}
Let $f(s)$ be a smooth function on $[a,b]\subset \mathbb R$ and $\mathcal S$ a Sturm-Liouville operator defined by $\mathcal S= -\sdd{}{s}+f$. Let $c$ and $d\in\mathbb R$ be such that $a\leq c< d\leq b$ and $f(c)\neq0$, and define $h$ as the unique solution to the initial value problem $\mathcal S h=0$, with $h(c)=1$, $h'(c)=0$. If $f$ and $h'$ do not vanish simultaneously on $[c,d]$, the number of negative eigenvalues of $\mathcal S$ on $\mathcal D^N([c,d])$  is
\be{ourIndex}
\textrm{Index}_{[c,d]^N} [\mathcal S]=\frac 1 2 \left (1- \textrm{Sign}[f(c)]\right)-  \sum_{\{\sigma_c\in(a,b]: h'(\sigma_c)=0\} } \textrm{Sign}\left(f(\sigma_c)\right).
\ee
\end{proposition}
\begin{proof}
Following Definitions~\ref{def-conj} and~\ref{def-index}, the number of negative inborn eigenvalues at $c$ is given by Proposition~\ref{prop-inborn}: there is one negative inborn eigenvalue if $f(c)<0$ and none if $f(c)>0$.  Then, define the set $\Sigma_c= \{\sigma_c\in(c,d]: h_1'(\sigma_c)=0\}$ of points $\sigma_c$ conjugated to $c$. Then, the problem $\mathcal S \tau=\lambda \tau$ on $\mathcal D^N([c,\sigma])$ with $\sigma \in [c,d]$ has a null eigenvalue if, and only if, $\sigma\in\Sigma_c$ (cf. Proposition~\ref{prop-NeuConj}). At each $\sigma_c$, an eigenvalue changes sign.   Dauge-Hellfer's theorem\re{DHN}  with $p=1$ and $q=f$ implies that a positive eigenvalue becomes negative when $f(\sigma_c)<0$. Indeed, by contradiction, the eigenfunction $u_{\sigma_c}$  cannot vanish at $\sigma_c$ otherwise it would be the unique and trivial solution of the initial value problem $\mathcal S u_{\sigma_c}=0$, $u_{\sigma_c}(\sigma_c)=u'_{\sigma_c}(\sigma_c)=0$. Similarly, a negative eigenvalue becomes positive if $f(\sigma_c)>0$. Note that $f$ cannot vanish at $\sigma_c$ since we assumed that $h'$ and $f$ do not vanish simultaneously.  We conclude that the number of negative eigenvalues is given by\re{ourIndex}. \qed
\end{proof}

For the operator $\mathcal S$ associated with the second variation of $\mathscr E$ at $\theta$, we have $f(s)=-{V_{\theta\theta}}\big|_{\theta(s)}$. Therefore, the direction of the  change of sign of the eigenvalue crossing 0 at conjugate points $\sigma_c$ is  given by the convexity of $V$ at $\theta(\sigma_c)$. Specifically, conjugate points for which $V$ is convex ($V_{\theta\theta}>0$)  transform positive eigenvalues into negative eigenvalues. \textit{Vice versa}, conjugate points for which $V$ is concave ($V_{\theta\theta}<0$) transform negative eigenvalues into positive eigenvalues. 

The existence of negative eigenvalues can be obtained by considering a particular set  of conjugate points $c\in(a,b)$ such that $\left . \sd{V}{\theta}\right|_{\theta(c)}=0$.
\begin{lemma}
\label{lem-IndexSub}
Let  $\theta:[a,b]\to \mathbb R$ be a solution of\re{thetapppot} such that $\nexists s\in[a,b]: \theta'(s)=V_\theta(\theta(s)) =0$, and $\gamma$
its associated trajectory in phase plane. Take $c\in(a,b)$ such that $\gamma(c)\in \Gamma_m \cup \Gamma_M$, and 
define the arcs  $\gamma_1:s\in[a,c)\to \gamma(s)$ and $\gamma_2:s\in(c,b]\to\gamma(s)$ with their  closures $\bar \gamma_1:s\in[a,c]\to \gamma(s)$ and $\bar \gamma_2:s\in[c,b]\to\gamma(s)$. Then,
\begin{itemize}
\item[(A)]
all points $\sigma_c\in[a,b]$ such that $\theta(\sigma_c)$ is a stationary point of $V$ are conjugated to one another;

\item[(B)]
the Index of $\theta$ on $\mathcal D^N([c,b])$ is given by
\be{}
\textrm{Index}_{\mathcal D^N([c,b])}[\theta]= 
\left \{
\begin{split}
J[\gamma_2] &\quad \textrm{if } \gamma(c) \in \Gamma_M, \\
J[\bar \gamma_2] &\quad \textrm{if } \gamma(c) \in \Gamma_m;
\end{split}
\right . 
\ee
\item[(C)]
the Index of $\theta$ on $\mathcal D^N([a,c])$ is given by 
\be{}
\textrm{Index}_{\mathcal D^N([a,c])}[\theta]= 
\left \{
\begin{split}
J[\gamma_1] &\quad \textrm{if } \gamma(c) \in \Gamma_M, \\
J[\bar \gamma_1] &\quad \textrm{if } \gamma(c) \in \Gamma_m.
\end{split}
\right . 
\ee
\end{itemize}
\end{lemma}
\begin{proof}
(A) The function $h_2(s)={\theta'(s)}/{\theta'(c)}$ solves\re{IVPN} with initial condition $h_2(c)=1;\quad h_2'(c)=0$. Proposition~\ref{prop-NeuConj} implies that
 the set of conjugate points to $c$ with respect to $\mathcal S$ on $\mathcal D^N([c,b])$ is $\widetilde C= \{\sigma_c\in(c,b]: h_2'(\sigma_c)=0\}$. Since $h_2'(\sigma)=\frac{1}{\theta'(c)} \theta''(\sigma)= \frac{-1}{\theta'(c)} V_\theta(\theta(\sigma))$, each element $\sigma_c\in \widetilde C$ is such that $\theta(\sigma_c)$ is a stationary point of $V$.

(B) Next, we compute the Index of $\theta$ on $\mathcal D^N([c,b])$.  According to (A), the eigenvalues of $\mathcal S$ on $\mathcal D^N([c,s])$ change sign when the solution $\theta(s)$ of\re{thetapppot} crosses stationary points of $V$. Statement $(B)$ then follows by  application of Proposition~\ref{prop-ourIndex} with $f(s)=\left . -\sdd{V}{\theta}\right|_{\theta(s)}$. Therefore, for any $\sigma_c \in \widetilde C$, we have
 \be{}
 \textrm{Sign} [f(\sigma_c)] = 
\left \{ 
\begin{split}
+1 \quad&\textrm{if } \gamma(\sigma_c) \in \Gamma_M,\\
-1 \quad&\textrm{if } \gamma(\sigma_c) \in \Gamma_m.\\
 \end{split}
 \right.
 \ee
  (C) follows trivially from the previous case with the change of variable $x=-s$.\qed
 \end{proof}

Let $a$, $b$, $c \in \mathbb R$ such that $a<c<b$ and $f\in C^1([a,b])$. 
Define the functional $\mK{f}{a}{b}$ on $\mathcal C^N([a,b])$ as
\be{DefK}
\mK{f}{a}{b}: \mathcal C^N([a,b])\to \mathbb R: \tau\to \int_a^{b} (\tau'(x))^2 + f(x) \tau^2(x) \textrm{d} x.
\ee
Let $\theta(s)$ be a solution of\re{thetapppot} with natural boundary conditions. Then, the  condition\re{secvar}  for the functional\ret{myfunc}{FormL}  is equivalent to 
\be{rephraseprop1}
\exists k>0: ~~\forall \tau \in \mathcal C^N([a,b]):~~ \mK{f}{a}{b} [\tau]>k \braket{\tau|\tau},\qquad 
\ee
where $f(s)=-\left . \sdd{V}{\theta}\right|_{\theta(s)}$. Further, if
\be{rephraseprop2}
\exists \tau \in \mathcal C^N([a,b]): \ \mK{f}{a}{b} [\tau]<0, 
\ee
then the stationary function $\theta$ is not minimal with respect to perturbations in $\mathcal C^N([a,b])$.

The strategy of the proof of Theorem~\ref{theo-IndJ} will be as follows: we
split the interval $[a,b]$ into $[a,c]$ and $[c,b]$ and address the minimality of the solution with respect to perturbations defined on $[a,c]$ and on $[c,b]$ separately. We first show that if $\theta$ is not minimal on both $\mathcal D^N([a,c])$ and $\mathcal D^N([c,b])$, then it is not minimal on $\mathcal D^N([a,b])$. Next, we  show that for any $c\in(a,b)$, if 
\be{}
\textrm{Index}_{[a,c]^N}[\theta]=\textrm{Index}_{[c,b]^N}[\theta]=0,
\ee
 then $\theta$ minimises $\mathscr E$ on $\mathcal D^N([a,c])$, on $\mathcal D^N([c,b])$ and also on $\mathcal D^N([a,b])$.  If $\theta$  is minimal on one of the subsets and not on the other, the method developed in the present section is inconclusive.

\begin{lemma}
\label{lem-unst}
If $\exists \tau_1 \in \mathcal C^N([a,c])$ and $\exists \tau_2 \in \mathcal C^N([c,b])$ such that $\mK{f}{a}{c}[\tau_1]<0$ and $\mK{f}{c}{b}[\tau_2]<0$, then $\exists \tau \in \mathcal C^N([a,b])$ such that $\mK{f}{a}{b}[\tau]<0$.
\end{lemma}
\begin{proof}
We need to distinguish three cases: 
\begin{itemize}
\item[(A)] $\tau_1(c)=\tau_2(c)$,
\item[(B)] $0\neq\tau_1(c)\neq \tau_2(c)\neq0$,
\item[(C)] $\tau_1(c)\neq \tau_2(c)$ but $\tau_i(c)=0$ with $i=1$ or $2$.
\end{itemize}
For each of these cases, we construct a function $\tau$ such that $\mK{f}{a}{b}[\tau]<0$.

In case (A), $\tau_A$ is
\be{taua}
\tau_A(x)=
\left \{
\begin{split}
\tau_1(x) &\qquad \textrm{if } x\in[a,c),\\
\tau_2(x) &\qquad \textrm{if } x\in[c,b],
\end{split}
\right .
\ee
and note that $\tau_A \in \mathcal C^N([a,b])$ by construction. Furthermore
\ben{}
\mK{f}{a}{b}[\tau_A] &=& \int_a^{b} (\tau_A'(x))^2 + f(x) \tau_A^2(x) \textrm{d} x
\nonumber\\
&=&  \int_a^{c} (\tau_A'(x))^2 + f(x) \tau_A^2(x) \textrm{d} x+\int_c^{b} (\tau_A'(x))^2 + f(x) \tau_A^2(x) \textrm{d} x\nonumber\\
&=&\mK{f}{a}{c}[\tau_1]+\mK{f}{c}{b}[\tau_2] <0.
\een

In case (B), define $\gamma=\frac{\tau_2(c)}{\tau_1(c)}$ such that the function 
\be{taub}
\tau_B(x)=
\left \{
\begin{split}
\gamma \tau_1(x) &\qquad \textrm{if } x\in[a,c),\\
\tau_2(x) &\qquad \textrm{if } x\in[c,b],
\end{split}
\right .
\ee
is continuous at $c$ and therefore $\tau_B\in\mathcal C^N([a,b])$ by construction. Then, 
\ben{}
\mK{f}{a}{b}[\tau_B] &=& \int_a^{b} (\tau_B'(x))^2 + f(x) \tau_B^2(x) \textrm{d} x\nonumber\\
&=&  \int_a^{c} \bigg(\gamma \tau_1'(x)\bigg)^2 + f(x)\,  \gamma^2\tau_1(x)^2 \textrm{d} x+\int_c^{b} \bigg(\tau_2'(x)\bigg)^2 + f(x)\, \tau_2(x)^2 \textrm{d} x\nonumber\\
&=&\gamma^2 \mK{f}{a}{c}[\tau_1]+\mK{f}{c}{b}[\tau_2] <0.
\een

In case (C), without loss of generality, assume $i=1$  and define $\nu=\tau_1+ \epsilon$ where $\epsilon \in \mathbb R$. Then, choosing $\epsilon$ such that
\be{}
0<\epsilon<\frac {|\mK{f}{a}{c}[\tau_1]|} 2 \textrm{Min} \left( \left \{\frac{1 }{2\int_a^c\Big| f(x) \tau_1(x) \Big|\textrm{d} x}, \frac{1}{\sqrt{\int_a^c| f(x)| \textrm{d}x}} \right\}\right),
\ee
we have
\ben{}
\mK{f}{a}{c}[\nu]&=& \int_a^{c} (\nu'(x))^2 + f(x) \nu^2(x) \textrm{d} x\nonumber\\
 &=& \mK{f}{a}{c}[\tau_1] + 2 \epsilon \int_a^c\Big[ f(x) \tau_1(x) \Big]\textrm{d} x+ \epsilon^2 \int_a^c f(x) \textrm{d}x<0.
\een
Then, we build $\tau_C$ as in case (B) by replacing $\tau_1$ by $\nu$. 
\qed
\end{proof}

\begin{lemma}
\label{lem-st}
Assume that $\exists M>0$ such that  $\forall\, \tau_1 \in \mathcal C^N([a,c]):~\mK{f}{a}{c}[\tau_1]>M \int_a^c {\tau_1}^2(s)~\textrm{d}s$ and $\forall\,  \tau_2 \in \mathcal C^N([c,b]):~\mK{f}{c}{b}[\tau_2]>M\int_c^b {\tau_2}^2(s)~\textrm{d}s$, then $\exists \bar M>0$ such that $\forall\,  \tau \in \mathcal C^N([a,b]):~\mK{f}{a}{b}[\tau]>\bar M \braket{\tau|\tau}$.
\end{lemma}
\begin{proof}
For $\tau\in \mathcal C^N([a,b])$ we define  
$p_\epsilon(x)=p_0 +p_1 \big(x-(c-\epsilon)\big) + p_2 \big(x-(c-\epsilon)\big)^2$ and $q_\epsilon(x)=q_0 +q_1 \big(x-(c+\epsilon)\big) + q_2 \big(x-(c+\epsilon)\big)^2$ where $p_0=\tau(c-\epsilon)$, $p_1=\tau'(c-\epsilon)$, $p_2=\frac{\tau'(c-\epsilon)}{2 \epsilon}$ and $q_0=\tau(c+\epsilon)$, $q_1=\tau'(c+\epsilon)$, $q_2=-\frac{\tau'(\epsilon)}{2 (c+\epsilon)}$ such that $p(c-\epsilon)=\tau(c-\epsilon)$, $p'(c-\epsilon)=\tau'(c-\epsilon)$, $p'(c)=0$ and $q(c+\epsilon)=\tau(c+\epsilon)$, $q'(c+\epsilon)=\tau'(c+\epsilon)$, $q'(c)=0$. 

Next, define the functions $\tau_{1\epsilon}$ and $\tau_{2\epsilon}$ as
\ben{}
\tau_{1\epsilon}(x)& =&
\left \{
\begin{split}
\tau(x) &\qquad \textrm{if } a\leq x< c-\epsilon,\\
p_\epsilon(x) &\qquad \textrm{if } c-\epsilon\leq x\leq c,\\
\end{split}
\right . \nonumber\\
 \tau_{2\epsilon}(x) &=&
\left \{
\begin{split}
q_\epsilon(x) &\qquad \textrm{if } c\leq x\leq c+\epsilon,\\
\tau(x) &\qquad \textrm{if } c+\epsilon< x< b,
\end{split}
\right .  \nonumber
\een
so that by construction, $\tau_{1\epsilon}\in \mathcal C^N([a,c])$ and $\tau_{2 \epsilon}\in \mathcal C^N([c,b])$.

Then, consider 
\ben{}
&&\mK{f}{a}{c}[\tau_{1\epsilon}]+ \mK{f}{c}{b}[\tau_{2\epsilon}] =\nonumber\\
&& \int_a^{c-\epsilon} (\tau'(x))^2 + f(x) \tau^2(x) \textrm{d} x+ \int_{c-\epsilon}^c \left(\sd{p_\epsilon}{x}\right)^2 + f(x) p_\epsilon^2(x) \textrm{d} x \nonumber\\
&&+ \int_{c}^{c+\epsilon} \left(\sd{q_\epsilon}{x}\right)^2 + f(x) q_\epsilon^2(x) \textrm{d} x 
+\int_{c+\epsilon}^{b} (\tau'(x))^2 + f(x) \tau^2(x) \textrm{d} x\nonumber\\
&=&\int_a^{c-\epsilon} (\tau'(x))^2 + f(x) \tau^2(x) \textrm{d} x +\int_{c+\epsilon}^{b} (\tau'(x))^2 + f(x) \tau^2(x) \textrm{d} x \nonumber\\
&&+ \epsilon \int_{0}^1 \bigg(\tau'(c-\epsilon)\bigg)^2 (1+z)^2 +  f(z) \Bigg(\bigg(\tau(c-\epsilon)\bigg)^2  + O(\epsilon)\Bigg) \textrm{d} z\nonumber\\
&&+\epsilon \int_{-1}^0 \bigg(\tau'(c+\epsilon)\bigg)^2 (1-w)^2 +  f(z) \Bigg(\bigg(\tau(c+\epsilon)\bigg)^2  + O(\epsilon)\Bigg) \textrm{d} w\nonumber\\
&=&\mK{f}{a}{b}[\tau] + O(\epsilon),\label{connectepsilon}
\een
where the second equality comes after the changes of variable $\epsilon z = x-(c-\epsilon)$ and $\epsilon w= x-(c+\epsilon)$ in the integrals involving $p_\epsilon$ and $q_\epsilon$ respectively. 

Since $\tau_{1\epsilon}\in \mathcal C^N([a,c])$ and $\tau_{2 \epsilon}\in \mathcal C^N([c,b])$, we have
\ben{}
\mK{f}{a}{c}[\tau_{1\epsilon}]+ \mK{f}{c}{b}[\tau_{2\epsilon}]&\geq& M \left(\int_a^c \tau_1^2(s)~\textrm{d}s +\int_c^b  \tau_2^2(s)~\textrm{d}s  \right) \nonumber\\
&=& M \left(\braket{\tau|\tau}  + O(\epsilon)\right). \label{connectepsilon2}
\een
Bringing\re{connectepsilon} and\re{connectepsilon2} together, we have 
\be{connectepsilonbringhome}
\mK{f}{a}{b}[\tau] \geq M \braket{\tau|\tau} + O(\epsilon). 
\ee
Since the left hand side and the first term in the right hand side of\re{connectepsilonbringhome} are independent of the  arbitrarily small $\epsilon$, for any $\bar M$ such that $0<\bar M<M$, we have $\mK{f}{a}{b}[\tau] \geq \bar M \braket{\tau|\tau}$.
\qed 
\end{proof}

In the following Lemma, we establish an identity between the two indices $I[\eta]$ and $J[\eta]$ defined in Section~\ref{sec-state}.
\begin{lemma}\label{lem-relIJ}
Let $\theta:[s_1,s_2]\to \mathbb R$ be a solution of\re{thetapppot1} such that $\nexists s\in[s_1,s_2]: \theta'(s)=0 \textrm{ and } \left . \sd{V}{\theta}\right|_{\theta(s)}=0$, and let $\eta$
its associated trajectory in phase plane. Then 
\be{relJI}
I[\eta] -1\leq J[\eta] \leq I[\eta]+1,
\ee
and
\be{botlimJ}
J[\eta]\geq-1.
\ee
\end{lemma}
\begin{proof} Let $f$ be a $C^2$ function of one variable over a connected domain $G$ and such that $f'$ and $f''$ do not vanish simultaneously. Consider the 
 functional $H$, that counts the number of minimal points minus the number of maximal points of $f$ over its domain $G$ 
\ben{}
H[f]&=&\# \{s\in G: f'(s)=0 \textrm{ and } f''(s)>0\} \nonumber\\
&&\qquad\qquad- \# \{s\in G: f'(s)=0 \textrm{ and } f''(s)<0\}.\label{defH}
\een
Since $f$ is $C^2$ over a connected domain, its minima and maxima are interspersed, $i.e.$  between any two consecutive minima (resp. maxima) there is one and only maximum (resp. minimum). Therefore, for such a function, we have
\be{Hgen}
-1\leq H[f] \leq 1. 
\ee 
The function $X:s\in[s_1,s_2]\to V(\theta(s))$ is $C^2 ([s_1,s_2])$ since $\theta\in C^2 ([s_1,s_2])$ and so is $V$. Accordingly\re{Hgen} implies that 
\be{HX}
-1\leq H[X]\leq 1.
\ee
Next we compute $H[X]$. First note that
\ben{Xp}
 X'(s)&=&\theta'(s) \,  \left . \sd{V}{\theta}\right|_{\theta(s)},\\
 X''(s)&=&\left (\theta'(s)\right)^2 ~ \left . \sdd{V}{\theta}\right|_{\theta(s)} - \left (\left . \sd{V}{\theta}\right|_{\theta(s)}\right)^2.\label{Xpp}
\een
According to\re{Xp},  any stationary point $z\in[s_1,s_2]$ of $X$ is due to the vanishing of either $\theta'(z)$ or $\left . \sd{V}{\theta}\right|_{\theta(z)}$. In the former case, which corresponds to $\eta(z)\in \Gamma_h$, $X''(z)<0$ according to\re{Xpp}. The latter case splits into two sub-cases: either $\left . \sdd{V}{\theta}\right|_{\theta(s)}>0$ so that $z$ is a minimum point and  $\eta(z)\in\Gamma_m$, or $\left . \sdd{V}{\theta}\right|_{\theta(z)}<0$ so that $z$ is a maximum point and $\eta(z)\in \Gamma_M$; remember that $\left . \sdd{V}{\theta}\right|_{\theta(z)}\neq 0$ since, by assumption, $V_\theta$ and $V_{\theta\theta}$ do not vanish simultaneously. Accordingly,
\ben{}
H[X] &=& \#\{ s\in [s_1,s_2]: \eta(s)\in\Gamma_m\} - \#\{ s\in [s_1,s_2]: \eta(s)\in\Gamma_M\}\nonumber\\
&&\qquad \qquad \qquad- \#\{ s\in [s_1,s_2]: \eta(s)\in\Gamma_h\}
\nonumber\\
&=& J[\eta]-I[\eta],\label{HIJ}
\een
and\re{relJI} follows from substituting\re{HIJ} in\re{HX}. 

Finally, the inequality\re{botlimJ} follows from\re{relJI} and the fact that $I[\eta]\geq0$ by definition. \qed
\end{proof}

\begin{proof}[Proof of theorem~\ref{theo-IndJ}]
For an arbitrary point $c\in(a,b)$, define the arcs $\gamma_1:s\in[a,c)\to \gamma(s)$ and $\gamma_2:s\in(c,b]\to\gamma(s)$ and their closure $\bar \gamma_1:s\in[a,c]\to \gamma(s)$ and $\bar \gamma_2:s\in[c,b]\to\gamma(s)$. 

First, consider the case $J[\gamma]<0$ which, by Lemma~\ref{lem-relIJ}, implies $J[\gamma]=-1$. Choose $c$ such that $\gamma(c) \in \Gamma_M$. Since $\gamma_1$ and $\gamma_2$ both exclude the maximum of $V$ at $\theta(c)$, we have
\be{splitJm1}
J[\gamma]= J[\gamma_1] + J[\gamma_2] -1=-1.
\ee
We also have, $J[\gamma_i]\geq0$, $i=1,2$, since, by contradiction, if $J[\gamma_i]=-1$, then $J[\bar \gamma_i]=-2$ which according to Lemma~\ref{lem-relIJ} is impossible. These inequalities together with\re{splitJm1} imply
\be{Ji0Jm1}
J[\gamma_1]=J[\gamma_2]=0.
\ee
Lemma~\ref{lem-IndexSub} then implies 
\be{Ind0sJm1}
\textrm{Index}_{\mathcal D^N([a,c])}[\theta]= \textrm{Index}_{\mathcal D^N([c,b])}[\theta]= 0.
\ee

Finally, Lemma~\ref{lem-IndexSub} and the assumption that both $\left . \sd{V}{\theta}\right|_{\theta(a)}\neq 0$ and $\left . \sd{V}{\theta}\right|_{\theta(b)}\neq 0$ imply that $c$ is conjugated neither to $a$ nor to $b$ so that the operator $\mathcal S$ has no null eigenvalue on $\mathcal D^N([a,c])$ or $\mathcal D^N([c,b])$. This fact together with Eq.\re{Ind0sJm1} implies that all eigenvalues of $\mathcal S$ on both $\mathcal D^N([a,c])$ and $\mathcal D^N([c,b])$ are strictly positive. It is therefore possible to choose a real number $M>0$ strictly less than both the smallest eigenvalues of $\mathcal S$ on $\mathcal D^N([a,c])$ and on $\mathcal D^N([c,b])$. By definition\re{DefK} of the functional $K$, we have $\forall\, \tau_1 \in \mathcal D^N([a,c]):~ \mK{f}{a}{c}[\tau_1]>M \int_a^c {\tau_1}^2\,\textrm{d}s$ and $\forall\,  \tau_2 \in \mathcal D^N([c,b]):~\mK{f}{c}{b}[\tau_2]>M\int_c^b {\tau_2}^2\,\textrm{d}s$. Finally, since $\mathcal D^N$ is dense in $\mathcal C^N$, we apply Lemma~\ref{lem-st} so that 
\be{}
\exists\bar M>0:~\forall \tau \in \mathcal C^N([a,b]):~\mK{f}{a}{b}[\tau]>\bar M \braket{\tau|\tau}.\ee
 The second variation of $\mathscr E$ is positive definite and $\theta$ is therefore a minimum.
 
Second, consider the case $J[\gamma]>0$. Choose the smallest possible $c\in(a,b)$ such that $J[\bar \gamma_1]=1$ and $\gamma(c)\in\Gamma_m$. Such a $c$ always exists since  $\lim_{c\to a } J[\bar \gamma_1]=0$ and $\lim_{c\to b} J[\bar \gamma_1]=J[\gamma]\geq1$ and $J[\bar \gamma_1]$ changes by increments of $\pm1$ while $c$ continuously spans $(a,b)$. We then have 
\be{splitJg0}
J[\gamma]= J[\bar \gamma_1] + J[\bar \gamma_2]-1>0,
\ee
where we have subtracted one because the first two terms counted the minimum of $V$ at $\theta(c)$ twice. Since $c$ was chosen so that $J[\bar \gamma_1]=1$, Eq.\re{splitJg0} implies that $J[\bar\gamma_2]>0$. 

Lemma~\ref{lem-IndexSub} then implies 
\be{IndJg0}
\textrm{Index}_{\mathcal D^N([a,c])}[\theta]=1\quad \textrm{and}\qquad \textrm{Index}_{\mathcal D^N([c,b])}[\theta]>0,
\ee
so that the operator $\mathcal S$ has at least one negative eigenvalue on both sub-domains. Let $\tau_1\in \mathcal D^N([a,c])$ and $\tau_2\in \mathcal D^N([c,b])$ be the corresponding eigenfunctions. By construction, 
$\mK{f}{a}{c}[\tau_1]<0$ and $\mK{f}{c}{b}[\tau_2]<0$ and, since $\mathcal D^N\subset\mathcal C^N$, Lemma~\ref{lem-unst} ensures that $\exists \tau\in  \mathcal C^N([a,b]):\mK{f}{a}{b}[\tau]<0 $. We conclude that  $\theta(s)$ is not minimal with respect to perturbations in $\mathcal C^N([a,b])$.\qed
\end{proof}

\begin{corollary}\label{cor-I2}
Let $\theta:[a,b]\to \mathbb R$ be a stationary function of the functional $\mathscr E$  such that $\left. \sd{V}{\theta}\right|_{\theta(s)}\neq 0$ on $s=a,b$ and $\nexists s\in[a,b]: \theta'(s)=0\textrm{ and }\left . \sd{V}{\theta}\right|_{\theta(s)}=0$. Let $\gamma$ be its associated trajectory in the phase plane. Then,
if $I[\gamma]\geq2$, $\theta$ is not a local minimiser of $\mathscr E$.
\end{corollary}
\begin{proof}
Substituting $I[\gamma]\geq2$ in Lemma~\ref{lem-relIJ} implies $J[\gamma]>0$ in which case Theorem~\ref{theo-IndJ} concludes that $\theta$ not minimal. \qed
\end{proof}

An important consequence of Corollary~\ref{cor-I2} is that, if $A=0$ in\re{thetapppot}, $\theta'(a)=\theta'(b)=0$ and therefore $I[\gamma]\geq 2$. Accordingly, all solutions of finite length such that $\nexists s\in[a,b]: \theta'(s)=0\textrm{ and }\left . \sd{V}{\theta}\right|_{\theta(s)}=0$ are unstable. We note that in such a case, if there exists a point $c \in [a,b]: \theta'(c)=0\textrm{ and }\left . \sd{V}{\theta}\right|_{\theta(c)}=0$ then the unique solution to the IVP\re{thetapppot1} with initial values $\theta(c)=\theta_c$ and $\theta'(c)=0$ is $\theta(s)=\theta_c$, $\forall s\in[a,b]$.  Corollary~\ref{cor-I2}  implies that, whenever $A=0$, the only stable stationary functions are constant. In that case, Proposition~\ref{cor-3} and Eq.\re{thetapppot1} imply that the only possible values of the constant are the maximal points of $V$.

 We are left with the  intermediate case, $J=0$, for which we can apply a different necessary condition for a stationary function to be minimal.

\begin{theorem}
\label{theo-xi}
Let $\theta$ be a stationary function of $\mathscr E$ such that  $\left . \sd{V}{\theta}\right|_{\theta(a)}\neq0$  and  $\left . \sd{V}{\theta}\right|_{\theta(b)}\neq0$. Define
\begin{equation}
\beta=\xi(a)-\xi(b),\qquad \textrm{where\ \ \ \ }\xi(s)=\theta'(s) \left.\sd{V}{\theta}\right|_{\theta(s)}.
\end{equation}
Then, if $\beta\leq 0$ the function $\theta$ is not minimal on $\mathcal C^N([a,b])$. 
\end{theorem}
\begin{proof}
We build a perturbation $\tau$ for which $\xi(a)\leq \xi(b)$ implies that\re{secvar2} is negative. Let $\nu$ and $\epsilon$ be strictly positive numbers such that $\epsilon<b-a$ and $\nu< b-a-\epsilon$. Define the polynomials 
\ben{}
g(s) &=& \theta'(a+ \epsilon) + \theta''(a+ \epsilon) \big(s- (a+\epsilon)\big) + \frac{\theta''(a+ \epsilon)}{2 \epsilon} \big(s- (a+\epsilon)\big)^2,\nonumber\\
h(s) &=& \theta'(b-\nu) + \theta''(b-\nu) \big(s- (b-\nu)\big) - \frac{\theta''(b-\nu)}{2 \nu} \big(s- (b-\nu)\big)^2.\nonumber
\een
Then consider the perturbation
\be{xiperturb}
\tau(s)= \left \{
\begin{split}
&g(s) \qquad \, \textrm{if  } s\in [a,a+\epsilon),\\
& \theta'(s) \qquad \textrm{if } s\in [a+ \epsilon,b-\nu],\\
&h(s)\qquad \textrm{if } s\in (b-\nu,b].
\end{split}
\right .
\ee
Note that by construction, $\tau\in \mathcal C^N([a,b])$. 
Substituting\re{xiperturb} in\re{secvar2} leads to 
\ben{}
\delta^2 \mathscr E_\theta [\tau]&=&\int_a^{a+\epsilon}\left (  
 \theta''(a+\epsilon) + \frac{\theta''(a+ \epsilon)}{ \epsilon} \big(s- (a+\epsilon)\big) 
 \right )^2 +f(s) 
 \left( g(s) \right)^2~\textrm{d} s\nonumber\\
 &&+ \int_{a+\epsilon}^{b-\nu} \left (\theta''(s)\right)^2+ f(s) \left(\theta'(s)\right)^2 \textrm{d} s \nonumber \\
&& + \int_{b-\nu}^b \left( \theta''(b-\nu) - \frac{\theta''(b-\nu)}{ \nu} \big(s- (b-\nu)\big) 
 \right )^2 +f(s) 
 \left( h(s) \right)^2~ \textrm{d} s. \nonumber
\een
After changing variable in the first and third lines according to $\epsilon z = s - (a+\epsilon)$ and $\nu t = s- (b-\nu)$ respectively, and integrating the first term of the second line by part, we obtain
\ben{}
&& \delta^2 \mathscr E_\theta [\tau] \nonumber\\
&=& \epsilon\, \left(\theta''(a+\epsilon)\right)^2 \int_{-1}^{0} \left (  
1+z\right)^2\textrm{d}z \nonumber\\
&&
+ \epsilon\,  \int_{-1}^{0} \left (f(a)  
+ O(\epsilon) \right) 
 \left( \theta'(a+\epsilon) + \theta''(a+\epsilon) \epsilon\,  \left(  z+ \frac{z^2}{2} \right)  \right)^2
\textrm{d} z
 \nonumber\\
&&+  \theta''(b-\nu) \theta'(b-\nu) - \theta''(a+ \epsilon) \theta'(a+ \epsilon) +\int_{a+\epsilon}^{b-\nu}\theta'(s) \overbrace{ \left (\theta'''(s) + f(s) \theta'(s)\right)}^0 \textrm{d} s \nonumber \\
&& +\nu\, \left (  \theta''(b-\nu) \right)^2 \int_{0}^1 \left( 1-t \right )^2 \, \textrm{d} t\nonumber\\
 && +\nu\, \int_0^1 \left(f(b) + O(\nu) \right) 
 \left( 
 \theta'(b-\nu) + \nu \theta''(b-\nu) \left( t - \frac{t^2}{2}\right) 
 \right)^2   \textrm{d} t\nonumber\\
 &&= 
\epsilon \left(\frac{\left(\theta''(a)\right)^2 }{3} - \left .  \sdd{V}{\theta}\right|_{\theta(a)} \left(\theta'(a)\right)^2 \right ) + O(\epsilon^2) \nonumber \\
&&\qquad + \theta''(b) \theta'(b) - \nu \left( \theta'''(b) \theta'(b) + \left(\theta''(b) \right)^2\right) + O(\nu^2) \nonumber\\
&&\qquad\qquad - \theta''(a) \theta'(a) - \epsilon \left( \theta'''(a) \theta'(a) + \left(\theta''(a)\right)^2\right) +O(\epsilon^2) \nonumber \\
&&\qquad + \nu \, \left ( \frac{\left(\theta''(b)\right)^2 }{3} -\left .  \sdd{V}{\theta}\right|_{\theta(b)} \left(\theta'(b) \right)^2 \right) + O(\nu^2)\nonumber\\
&&= \overbrace{\left . \sd{V}{\theta}\right|_{\theta(a)}  \theta'(a)}^{\xi(a)} - \overbrace{ \left . \sd{V}{\theta}\right|_{\theta(b)}  \theta'(b) }^{\xi(b)} -\epsilon ~\frac{2}{3} \left ( \left . \sd{V}{\theta}\right|_{\theta(a)} \right)^2 - \nu  ~ \frac{2}{3} \left ( \left . \sd{V}{\theta}\right|_{\theta(b)} \right)^2\nonumber\\
&&
\qquad\qquad\qquad\qquad\qquad\qquad \qquad\qquad\qquad\qquad+O(\epsilon^2)+ O(\nu^2). \label{secvaxi}
\een
Accordingly if $\xi(a)-\xi(b)<0$, then for $\epsilon$ and $\nu$ sufficiently small, $\delta^2 \mathscr E_\theta$  
is negative for the pertubation\re{xiperturb}. If $\xi(a)=\xi(b)$, the second variation is dominated by the linear terms in\re{secvaxi} which is negative so that once again, the second variation applied on the perturbation\re{xiperturb} is negative for sufficiently small $\epsilon$ and $\nu$. \qed
\end{proof}

\section{Application\label{sec-ex}}

As an example, we study the case of a planar weightless inextensible and unshearable rod of length $L$. The rod is pinned at one end and a weight of mass $m$ is attached at the other end (see Fig.~\ref{fig-a}). The rod is uniformly curved in its reference state with a reference curvature $\widehat{u}>0$. Depending on the parameters, this system supports multiple equilibrium solutions and we can apply our general results to determine their stability.  We  show that the stability of most (but not all) equilibria can be decided by Theorem~\ref{theo-IndJ}. For this particular example, Theorem~\ref{theo-xi} was sufficient to prove that all cases where $J[\gamma]=0$ are actually unstable.

The potential energy $\mathcal  E$ of the system is
\be{EnEx}
\mathcal  E= \int_0^L \frac{EI}{2}  \left(\sd{\theta}{s}-\widehat{u}\right)^2\textrm{d}s + mg \int_0^L \cos \theta(s)~ \textrm{d}s,
\ee
where the arc length $s$ is used to parameterise the rod ($s=0$ at the frame and $s=L$ at the massive point), $\theta$ is the angle between the tangent to the rod (towards increasing $s$) and the upward  vertical. The first term in\re{EnEx} is the elastic energy of the rod where we have assumed linear constitutive laws for the moments, $(EI)$ is the bending stiffness of the rod classically estimated as the product of the Young's modulus $E$ by the second moment of area $ I$ of the rod's section. The second term accounts for the potential energy of the massive point where $g$ is the acceleration of gravity. 

We non-dimensionalise the problem according to
\begin{equation}
x=\frac{s}{L},\quad v=\frac{\widehat{u}^2 (EI)}{2mg},\quad M=\frac{m g L^2}{EI},\quad \mathcal E= \mathscr E \frac{EI}{L}.
\end{equation}
In its reference (unstressed) state, the rod is a multi-covered ring with curvature $\sqrt{2 M v}/L$ and  $n_{\textrm{loop}}=\sqrt{2M v}/(2\pi)$ full loops. In these new variables, Eq.\re{EnEx} becomes 
\be{EnExND}
\mathscr  E= \int_0^1\frac{ \left(\theta'-\sqrt{2 M v}\right)^2}{2} + M \cos \theta \textrm{d}x,
\ee
where  $(\ )'=\sd{(\ )}{x}$. We note that\re{EnExND} has the form\ret{myfunc}{FormL} with
\be{defVappli} 
V=-M \cos\theta.
\ee
The system depends on two parameters: $v$ measuring the reference curvature of the rod in units of $\sqrt{ 2 mg/(EI)}$ and $M$ measuring the mass of the attached weight in units of $\frac{EI}{g L^2}$. Also note that increasing $L$ increases $M$ but leaves $v$ unchanged.

\subsection{Classifying the equilibria\label{sec-class}}

The equilibria of the system are the solutions to the Euler-Lagrange problem for\re{EnExND}:
\be{ELwS}
\theta'' + M \sin \theta=0,\qquad \theta'(0)=\theta'(1)=\sqrt{2 M v}. 
\ee
The potential energy of the system is  given by\re{EnExND}, but the following pseudo-energy is conserved along the length of the rod: 
\be{pE}
\frac{(\theta')^2}{2} + V(\theta) = C.
\ee
This pseudo-energy is a convenient quantity to classify the solutions from the point of view of the `point-mass-in-a-potential' analogy.
Since initial value problems have unique solutions, different solutions of\re{ELwS} must have different initial point $\theta(0)$ and their respective pseudo-energies 
\be{pEinit}
C=M v- M \cos \theta(0)
\ee
 must therefore be different -- up to a global rotation of the whole rod by $2 \pi$. 
 
 Further, since the potential $V$ is a multiple of $M$, we can express all energies in units of $M$: $C= e M$. For fixed $M$ and $v$, the solutions of\re{ELwS} can therefore be conveniently labelled by their non-dimensional pseudo-energy $e$. The rescaled version of\re{pEinit} reads
 \be{pEnND}
 e=v - \cos \theta(0).
 \ee
Consequently, for a given reference curvature $v$, the pseudo-energy of all solutions is constrained: $e\in[v-1,v+1]$.  From\re{pE} and the boundary condition at $x=1$ in\re{ELwS}, all solutions must start and end at the same height: $V(\theta(0))=V(\theta(1))$. Therefore, we can classify the solutions of\re{ELwS} according to five  categories
\begin{figure}
  \centering
  {\includegraphics[width=\linewidth]{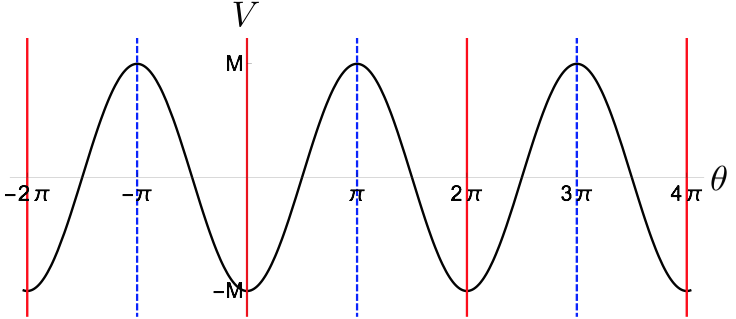}
  }
  \caption{The potential function $V=-M \cos \theta$. The full vertical lines (red online) show min-boundaries and the dashed vertical lines (blue online) show max-boundaries.\label{fig-pot}}
\end{figure}
\begin{itemize}
\item[(a)] Solutions that go over one and only one maximum of $V$ and do not cross any minima.
\item[(b)] Solutions that stay within one well of $V$ and oscillate therein: $V(\theta(x))< M~~\forall x\in[0,1]$. 
\item[(c)] Solutions that cross $k$ minima and $k$ maxima of $V$ with $k\in\mathbb N_0$. 
\item[(d)] Solutions that cross $k$ minima and $k+1$ maxima of $V$ with $k\in\mathbb N_0$. 
\item[(e)] Solutions that cross $k+1$ minima and $k$ maxima of $V$ with $k\in\mathbb N_0$. 
\end{itemize}
Example of equilibria in these five categories are shown in Figs.~\ref{fig-a}-\ref{fig-e}.

By direct application of Theorem~\ref{theo-IndJ}, all equilibria of the categories $(a)$ and $(d)$ are stable and all equilibria of categories $(b)$ and $(e)$ are unstable. Indeed, for category $(a)$ an example of trajectory $\gamma$ in the phase space is shown in bold (red online) in Fig.~\ref{fig-a}$(c)$.  The only boundary crossed by the trajectory $\gamma$ is the max-boundary at $\theta=\pi$, hence $J[\gamma]=-1$ and Theorem~\ref{theo-IndJ} implies that the solution is stable. 

Solutions in category $(b)$ cross (sometimes multiple times) the min-boundary at $\theta \textrm{ mod } 2\pi=0$, $J[\gamma]>0$ and Theorem~\ref{theo-IndJ} implies that these equilibria are unstable. 

Equilibria in category $(d)$ alternatively cross max- and min-boundaries of $V$ so that they first and last cross max-boundaries. For such solutions $J[\gamma] = (k) - (k+1)= -1$ and from Theorem~\ref{theo-IndJ} we conclude that any such equilibrium is stable. 

Equilibria in category $(e)$ alternatively cross min- and max-boundaries of $V$ so that they first and last cross min-boundaries. For such solutions $J[\gamma] = (k+1) - (k)= 1$ and from Theorem~\ref{theo-IndJ} we conclude that any such equilibrium is unstable. 

Finally, equilibria in  category $(c)$ provide examples for which Theorem~\ref{theo-IndJ} is inconclusive since $J[\gamma]=k -k=0$. However because of the periodicity of $V$, and the fact that each equilibria must respect $V(\theta(a))=V(\theta(b))$, we have $\xi(a)=\theta'(a) \left.\sd{V}{\theta}\right|_{\theta(a)} = A\left.\sd{V}{\theta}\right|_{\theta(a)}= A\left.\sd{V}{\theta}\right|_{\theta(b)}= \theta'(b) \left.\sd{V}{\theta}\right|_{\theta(b)} =\xi(b)$ and Theorem~\ref{theo-xi} implies that all such equilibria are unstable. 

Gathering these results, for this particular system, a stationary solution is stable whenever $\theta(0) \mod 2\pi \in(0,\pi)$ and $\theta(L)\mod 2 \pi \in(\pi, 2\pi)$. 

\subsection{Detailed analysis of the equilibria\label{sec-det}}

\paragraph{Category $(a)$} corresponds to solutions which come down from the frame arching in the same direction (but not by the same amount) as the reference curvature of the rod and without looping. When there is a solution in this category, it is unique modulo $2 \pi$. We note that for such a solution to exist,  it must be possible to fit the whole length of the rod between two minima of $V$ (at 0 and $ 2 \pi$ in Fig.~\ref{fig-a}). In other words, the length of rod $L_{\textrm{open}}$ that would be required to go from one minimum to the next must be greater than 1: 
\be{Lopen}
1<L_{\textrm{open}}= \int_0^{L_{\textrm{open}}} \textrm{d} x = \int_{0}^{2\pi} \frac{1}{\theta'} \textrm{d} \theta
 = \frac{1}{\sqrt{2M}}  \int_{-\pi}^\pi \frac{1}{\sqrt{v-1+ \cos \theta}} \textrm{d} \theta =  \frac{4\, K(2/v)}{\sqrt{2M v}} ,
 \ee
where, we used\re{pEinit} to express $\theta'$ and where $K(m)=\int_0^{\pi/2} \frac{\textrm{d}\phi}{\sqrt{1-m\sin^2 \phi}}$ is the complete elliptic integral of the first kind. For a given $M$, Eq.\re{Lopen} can be inverted to obtain the maximal $v$ for which an open solution exists. 

\begin{figure}
  \centering
  {\includegraphics[width=\linewidth]{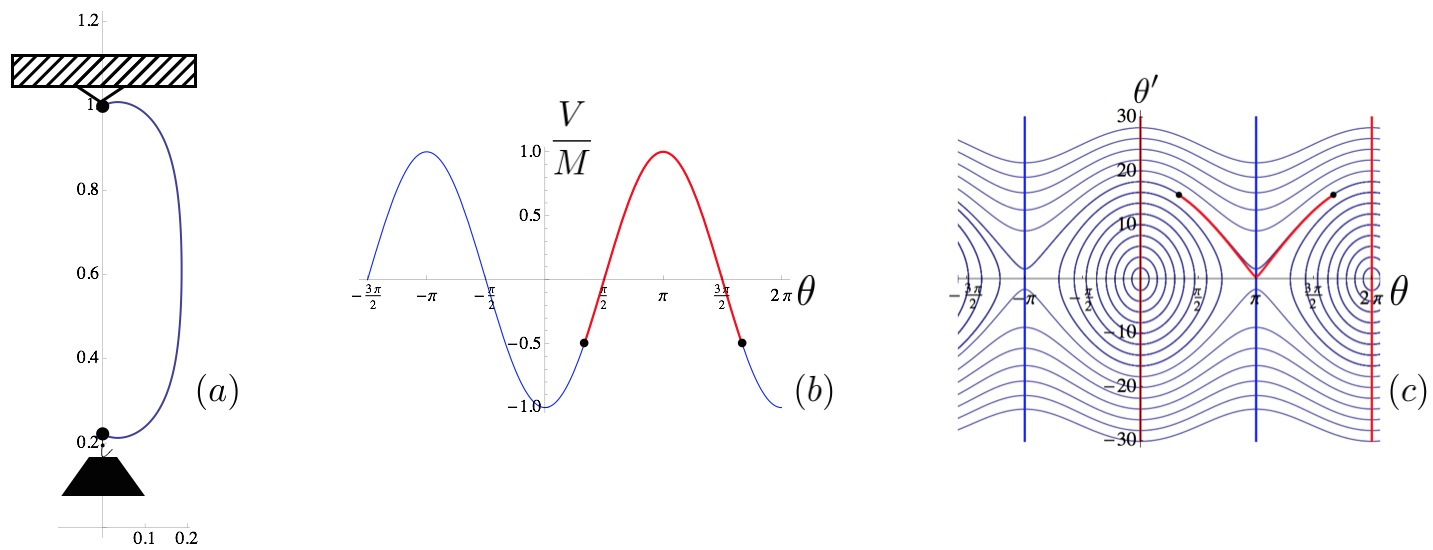}
  }
  \caption{The equilibrium of category $(a)$ in the case $M=81$ and $v=1.5$. Substituting Eq.\re{sola} in the expression\re{EnExND} of the total energy of the system, the total energy of this solution is $\mathscr E_a\simeq11.22$. $(a)$: The physical realisation of the solution. $(b)$: Shape of the potential $V$ normalised by $M$ (thin and blue online) as well as the part of it that is covered by the solution (thick and red online) which starts and ends at the black dots.  $(c)$: Phase space of the problem in general (thin and blue online) and of this particular solution (thick and red online). The vertical lines (red online) at $2 z \pi$  and the vertical lines (blue online) at $\pi+ 2 z\pi$  with $z \in \mathbb Z$ are respectively the min- and max- boundaries $\Gamma_m$ and $\Gamma_M$. This equilibrium is \textbf{stable}.
  \label{fig-a}}
\end{figure}

For a fixed choice of $M$ and $v$ and providing $L_{\textrm{open}}>1$, the pseudo-energy $e_a$ of the solution $(a)$ can be computed by inverting
\ben{}
1= \int_0^1 \textrm{d} x &=& \frac{1}{\sqrt{2M} }\int_{\arccos (v-e_a)}^{2 \pi -\arccos(v-e_a)} \frac{\textrm{d} \theta }{\sqrt{e_a+\cos\theta}} \nonumber\\
&=&\frac{4}{\sqrt{2M (1+e_a)} }  \int_{\frac{\arccos (v-e_a)}{2}}^{\frac{\pi}{2}} \frac{\textrm{d}\phi }{\sqrt{1- \frac{2}{1+e_a} \sin^2\phi}} \nonumber \\
&=&
\frac{2}{\sqrt{M}}~~ \ell_a(e_a |v), \label{finda}
\een
where we used\re{pEnND} to express $\theta(0)$ in the second equality together with the fact that the solution is symmetric about $\theta=\pi$ and where $F(\phi|m)=\int_0^\phi \frac{\textrm{d}\phi}{\sqrt{1-m\sin^2 \phi}}$ is the elliptic integral of the first kind and where
\be{Defella}
\ell_a(e |v)=
 \sqrt{\frac{2}{1+e} } \left(K \left(\frac{2}{1+e}\right) - F\left (\frac{\arccos(v-e)}{2}\Big | \frac{2}{1+e} \right) \right).
 \ee
The function $\ell_a$ together with similar functions for the other categories prove very useful to find all equlibria for a given $M$ and $v$. 

For a given root of Eq.\re{finda}, a similar argument can be used to compute the solution $(a)$ as
\be{sola}
\theta(s)=2\, \textrm{am} \Bigg( K\left(\frac{2}{1+e_a}\right)+  \frac{\sqrt{2M(1+e_a)}} {2} ({s-1/2} )  ~\Bigg |~ \frac{2}{1+e_a}    \Bigg),
\ee
where $\textrm{am}(u|m)$ is the Jacobi amplitude of the elliptic integral of the first kind: the inverse of the function $F(\phi|m)$. All these solutions are \textbf{stable}.

\begin{figure}[]
  \centering
  {
  \includegraphics[width=\linewidth]{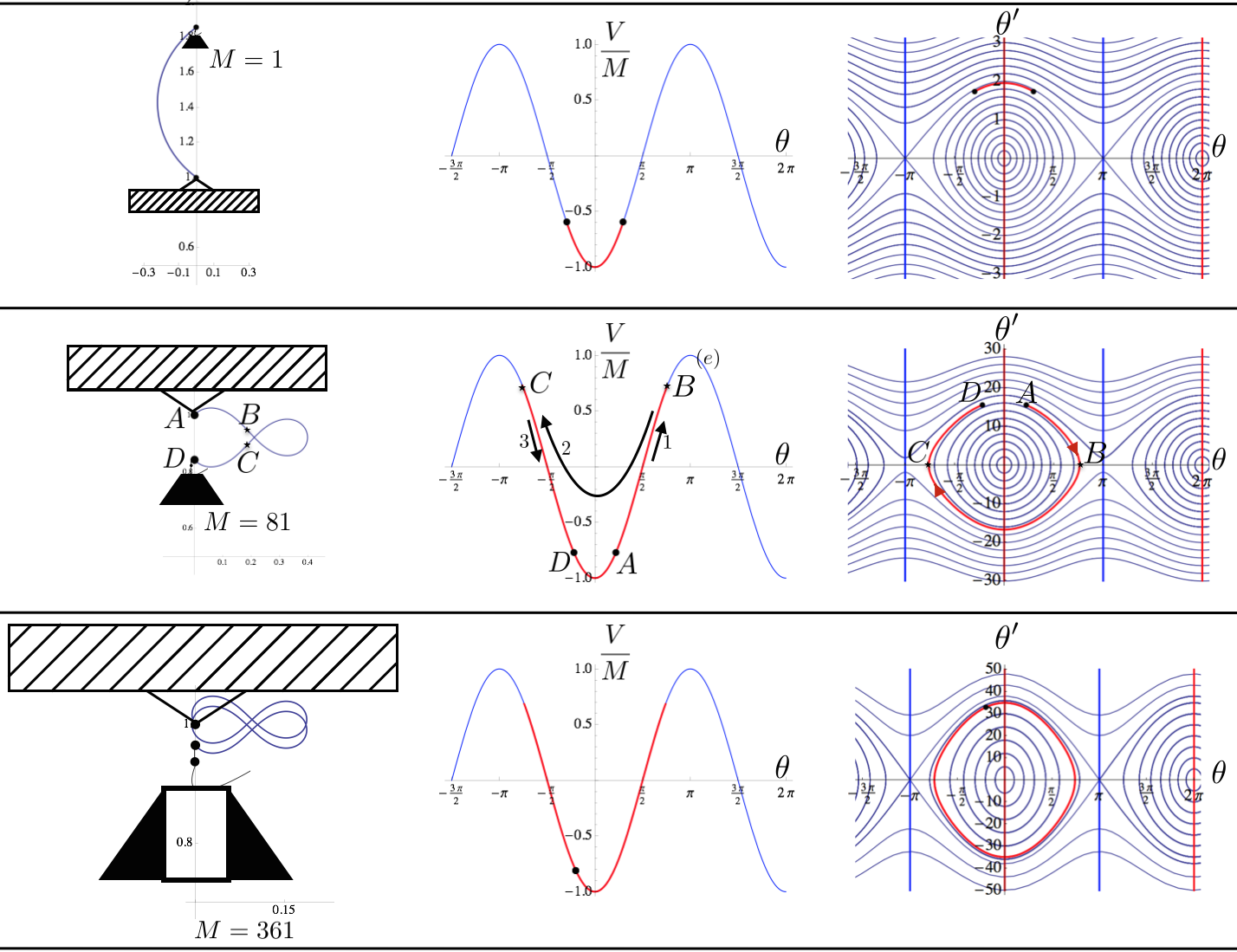}
   }
  \caption{Examples of equilibria in category $(b)$ with $v=1.5$ and $M=1$ (top), $M=81$ (middle) and $M=361$ (bottom). At $M=1$, only a single simple oscillation exists in category $(b)$. At $M=81$ a single complex oscillation exists together with a full cycle ($k=1$). For $M=361$, we  show the case of two full cycles in the figure ($k=2$) but there also exists a solution with $k=1$. Since the middle example has the same parameters as the case showed in Fig.~\ref{fig-a}, it is interesting to compare the total energy in the system for the two equilibria. We find: $\mathscr E_b\simeq176.96$. These equilibria are \textbf{unstable}.
   \label{fig-b}}
\end{figure}
\paragraph{Category $(b)$} contains a variety of solutions that can be further sub-divided into different oscillatory modes. We will dispense from a detailed study of all possibilities as we proved that all these equilibria are \textbf{unstable}.  It can be divided in two categories depending on whether there is a single or multiple oscillations in the potential well. 

In the case of a single oscillation there is a further sub-division according to whether the sign of $\theta(0)$ matches the sign of $\widehat u$ or not. If not (see Fig.~\ref{fig-b}(a)), the pseudo-energy is given by the requirement that the length 
\ben{}
L_{\textrm{simple swing}} &=& 2 \int_0^{\theta(1)} \frac{\textrm{d}\theta}{\theta'}= \frac 2 {\sqrt{M}} \sqrt{\frac 2 {1+e}} \int_0^{\frac{\arccos (v-e)}2} \frac{\textrm{d} \phi }{\sqrt{1-\frac  2 {1+e} \sin^2 \phi}}\nonumber\\
&=&\frac 2 {\sqrt{M}} \sqrt{\frac 2 {1+e}} F\left (\frac{\arccos (v-e)}{2} \big | \frac 2 {1+e} \right ), \label{LengthSimpleSwing}
\een
must be equal to 1. Let us define the function $\ell_b$ for future use and similarly to $\ell_a$ in\re{Defella}:
\be{}
\ell_b(e|v) = \sqrt{\frac 2 {1+e}} F\left (\frac{\arccos (v-e)}{2} \big | \frac 2 {1+e} \right ).
\ee

If the signs of $\theta(0)$ and $U$ match, a somewhat more complex equilibrium occurs (see Fig.~\ref{fig-b} (Middle)). Its length can be computed as 
\be{}
L_{\textrm{complex swing}}=\frac 2 {\sqrt{M}} \big( 2 \ell_b(e|0) - \ell_b(e|v) \big),
\ee
and the corresponding $e$ is found by requiring that $L_{\textrm{complex swing}}=1$.

For multiple oscillations to occur it must be possible to have solutions with pseudo-energy $e<1$ (otherwise the solution leaves the well of potential $V$). This gives a boundary on the reference curvature for which this can happen indeed $v-1<e<1$ implies $v<2$.  Let us simply note that there may be multiple (and in fact many) solutions of this type. To prove this we can simply compute the length of rod required to do a half oscillation from $-\theta_e=-\arccos[-e]$ to $\theta_e$ (where $\theta_e$ is defined as the angle at which $\theta'=0$ for that particular value of pseudo-energy):
\be{Lswing}
L_{\textrm{swing}}(e)= \int_{-\theta_e}^{\theta_e} \frac{\textrm{d}\theta}{\theta'} = \frac{4}{\sqrt{2 M} }\int_0^{\theta_e/2} \frac{\textrm{d}\phi}{\sqrt{e+\cos 2\phi}}=\frac{2}{\sqrt{ M}}~\ell_b(e|0).
\ee
If there exists a value of $e$ such that $L_{\textrm{swing}}(e)= 1/(2 k)$ with $k\in\mathbb N_0$, there exists a solution starting at\footnote{That is assuming that $U>0$, if $U<0$ the solution starts at $\arccos[v-e]$.} $\theta(0)=-\arccos[v-e]$, oscillating $k$ times and ending with $\theta(1)=\theta(0)$. We observe that $L_{\textrm{swing}}(e)$ is a monotonically increasing function of $e$ and that $\lim_{e\to 1} L_{\textrm{swing}}(e)=\infty$ so the maximum number of oscillations for a given $M$ and $v$ is given by $\max \Big[k\in \mathbb N_0: k< \frac 1{2 L_{\textrm{swing}}(v-1)} \Big]$. Furthermore, it is easy to show that $\lim_{e\to -1} L_{\textrm{swing}}(e) =\frac{\pi}{\sqrt{M}}$ so that the maximum number of oscillations for a given $M$  happens when $v=0$ (so that $e=-1$ can be asymptotically reached) and is given by $\max \Big[ k\in \mathbb N_0: k < \frac{\sqrt{M}}{2 \pi} \Big]$. Finally, because of the monotonicity of $L_{\textrm{swing}}(e)$, if there exist a solution with $k$ oscillations, there exist solutions with $j$ oscillations for all $j: 1\leq j\leq k$. 

To summarise, in this category, there exist
\begin{itemize}
\item an equilibrium with a single simple swing if $\ell_b(e|v) = \sqrt{M}/2$ has a solution in $e\in[v-1,v+1]$,
\item an equilibrium with a single complex swing if $2 \ell_b(e|0) - \ell_b(e|v) = \sqrt{M}/2$ has a solution in $e\in[v-1,v+1]$,
\item equilibria performing $k$ oscillations for all $k$ for which $2 k \,\ell_b(e|0)=\sqrt{M}/2$ has a solution in $e \in[v-1,1]$.
\end{itemize}

\begin{figure}[]
  \centering
  {
  \includegraphics[width=\linewidth]{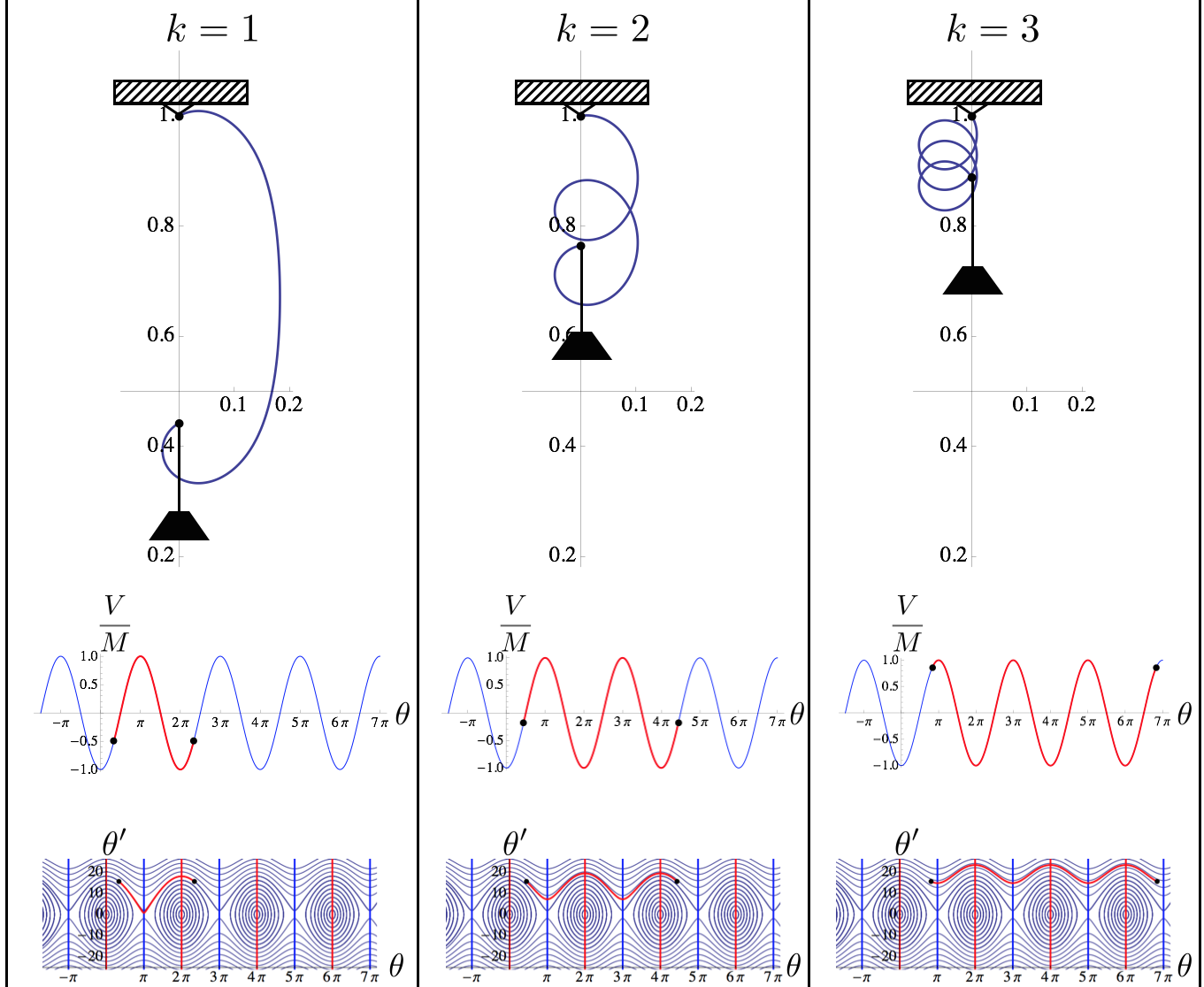}
   }
  \caption{Examples of equilibria in category $(c)$ with $v=1.5$ and $M=81$ for which there are three solutions to Eq.\re{findc}: one for each $k=1,2,3$. The energy of the system is $\mathscr E_{c1}\simeq14.60$, $\mathscr E_{c2}\simeq-5.37$, and $\mathscr E_{c3}\simeq0.74$, respectively. These equilibria are \textbf{unstable}.
   \label{fig-c}}
\end{figure}

\paragraph{ Category $(c)$} is defined by solutions with exactly $k$ loops between the frame and the weight. Their existence depends only on  the length $L_{\textrm{Loop}}$ required for the solution to loop (go from any $\theta^\star$ to $\theta^\star+ 2 \pi$) for given values of the mass $M$ and the pseudo-energy of the solution $e$. 
This length can be easily computed as 
\ben{}
L_{\textrm{loop}}(e)&=& \int_0^{L_{\textrm{loop}}} \textrm{d} x = \int_{-\pi}^\pi \frac{1}{\theta'} \textrm{d} \theta
 = \frac{1}{\sqrt{2M}}  \int_{-\pi}^\pi \frac{1}{\sqrt{e+ \cos \theta}} \textrm{d} \theta \nonumber\\
 &=& \frac{2}{\sqrt{M}} \sqrt{\frac{2}{1+e}}  K\left ( \frac{2}{1+e}\right). \label{Lloop}
\een

There exists a solution in category $c$ for each value of $k \in\mathbb N_0$ for which $k\, L_{\textrm{loop}}(e) =1$ has a solution $e\in[v-1,v+1]$.  We therefore define
\be{Defellc}
\ell_c (e|k) = k~ \sqrt{\frac{2}{1+e}}  ~K\left ( \frac{2}{1+e}\right),
\ee
and there is a solution in category $(c)$ for each $k$ such that 
\be{findc}
\ell_c (e|k) = \frac{\sqrt{M}}{2},
\ee
has a solution $e\in[v-1,v+1]$.  All these solutions are \textbf{unstable}.

\begin{figure}[]
  \centering
  {
  \includegraphics[width=\linewidth]{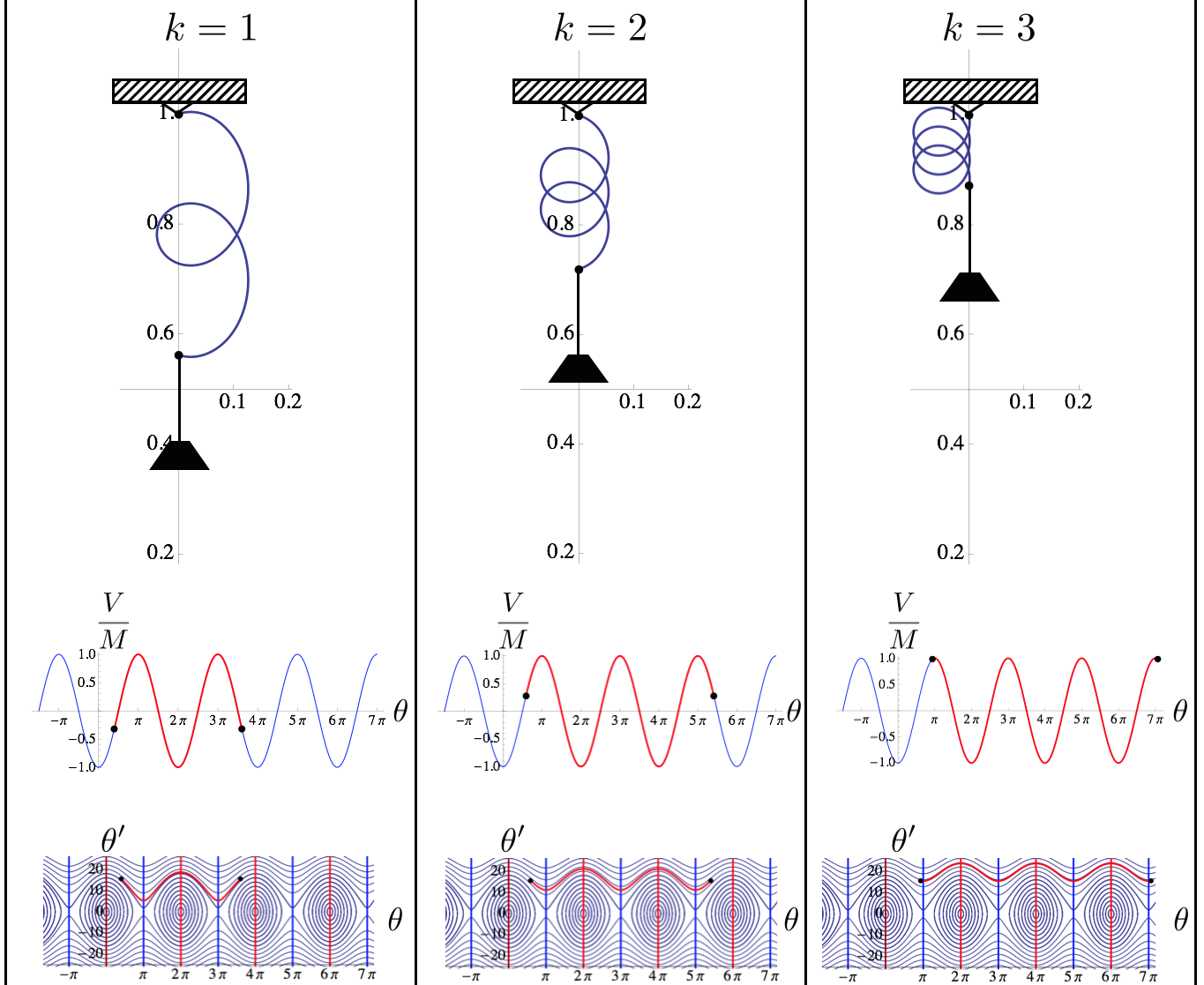}
   }
  \caption{Examples of equilibria in category $(d)$ with $v=1.5$ and $M=81$ for which there are three solutions to Eq.\re{findd}: one for each $k=1,2,3$. The energy of the system is $\mathscr E_{d1}\simeq-11.35$, $\mathscr E_{d2}\simeq-16.74 $, and $\mathscr E_{d3}\simeq0.58$, respectively. These equilibria are \textbf{stable}.  
   \label{fig-d}}
\end{figure}

\paragraph{ Category $(d)$} gathers solutions that loop $k$ times between the frame and the weight and arc over the next maximum. Their existence depends on the length $L_{\textrm{Loop}}$ computed in\re{Lloop} and on the length $L_{\textrm{arch}} = \frac 2 {\sqrt{M}} \ell_a(e|v)$ required to cover the extra arch from $2 k \pi+\theta(0)$ to $(2k+1) \pi - \theta(0)$ (assuming that $\theta(0)\in[0, \pi]$). 

The pseudo-energy of a solution with $k$ loops is then specified by the condition $L_{\textrm{arch}} + k L_{\textrm{loop}} =1$ which is equivalent to $\ell_d (e|v,k) = \frac{\sqrt{M}} 2$ where the function $\ell_d$ is defined by
\be{findd}
\ell_d(e|v,k)= \sqrt{\frac 2 {1+e}}  \bigg[  K\left (  \frac 2 {1+e}  \right ) (1+k) - F\left ( \frac{\arccos (v-e)}{2} \big |  \frac 2 {1+e} \right ) \bigg],
\ee
and there is a solution in category $(d)$ for each $k$ such that 
\be{findd}
\ell_d (e|v,k) = \frac{\sqrt{M}}{2},
\ee
has a solution $e\in[v-1,v+1]$. All these solutions are \textbf{stable}.

\begin{figure}[]
  \centering
  {
  \includegraphics[width=\linewidth]{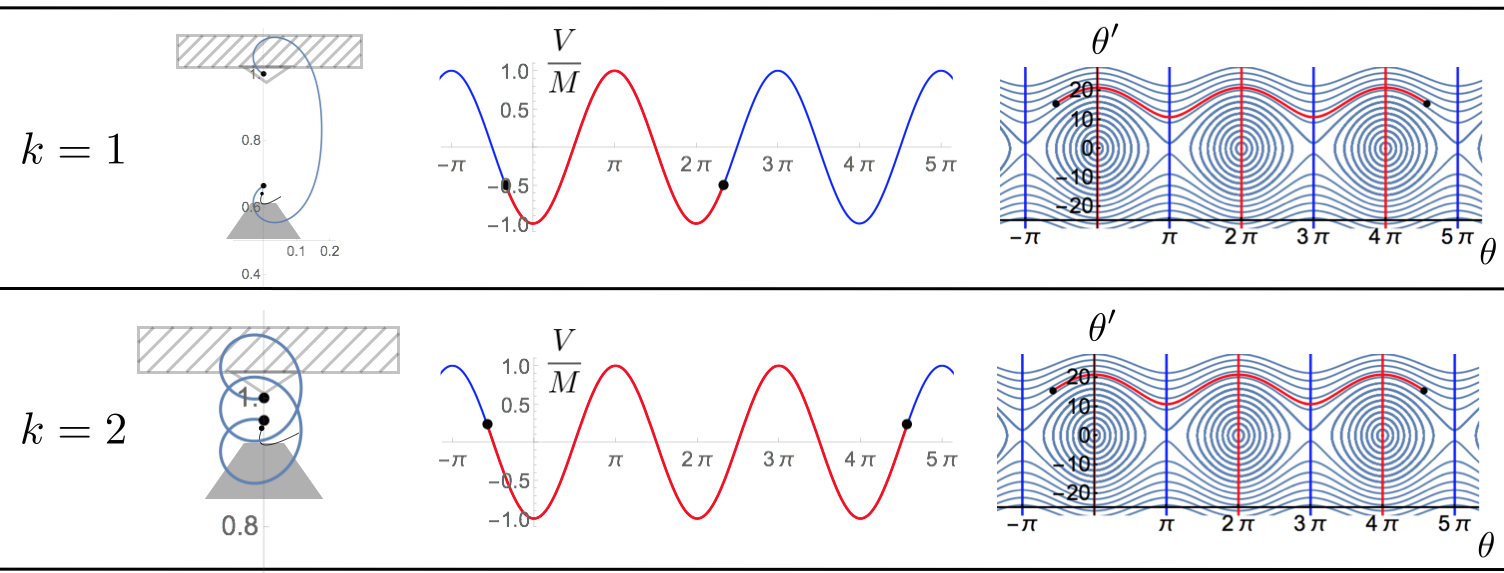}
   }
  \caption{Examples of equilibria in  category $(e)$ with $v=1.5$ and $M=81$ for which there are two solutions to Eq.\re{finde}: one for each $k=1,2$. The energy of the system is respectively: $\mathscr E_{e1}\simeq 18.01$ and $\mathscr E_{e2}\simeq 3.73 $. These equilibria are \textbf{unstable}.
   \label{fig-e}}
\end{figure}

\paragraph{ Category $(e)$} gathers solutions that loop $k$ times between the frame and the weight and swing across the next minimum. Their existence depends on the length $L_{\textrm{Loop}}$ computed in\re{Lloop} and on the length $L_{\textrm{simple swing}}$ computed in\re{LengthSimpleSwing} and required to cover the extra swing from $2 k \pi+ \theta(0)$ to $2 k \pi -\theta(0)$ (assuming that $\theta(0)\in[-\pi,0]$).

The pseudo-energy of a solution with $k$ loops is then specified by the condition $ L_{\textrm{simple swing}} + k L_{\textrm{loop}} =1$ which is equivalent to $\ell_e (e|v,k) = \frac{\sqrt{M}} 2$ where the function $\ell_e$ is defined by
\be{defelle}
\ell_e(e|v,k)= \sqrt{\frac 2 {1+e}}  \bigg[ k~ K\left (  \frac 2 {1+e}  \right )  + F\left ( \frac{\arccos (v-e)}{2} \big |  \frac 2 {1+e} \right ) \bigg],
\ee
and there is a solution in category $(e)$ for each $k$ such that 
\be{finde}
\ell_e (e|v,k) = \frac{\sqrt{M}}{2},
\ee
has a solution $e\in[v-1,v+1]$. All these solutions are \textbf{unstable}.

\paragraph{All equilibria} can be summarised  in a single figure. Indeed, for each category, it is possible to obtain an equation of the form $\ell(e|\cdots)=\frac{\sqrt{M}}{2}$ for the pseudo-energy of the solution. In Fig.~\ref{fig-gather}, we plot these non-dimensional solution lengths as functions of $e$. For a given mass $M$ and reference curvature of the spring $v$, the possible equilibria are simply determined by the intersection of these curves and the horizontal (thick black) at $\sqrt{M}/2$.

\begin{figure}[]
  \centering
  {
  \includegraphics[width=\linewidth]{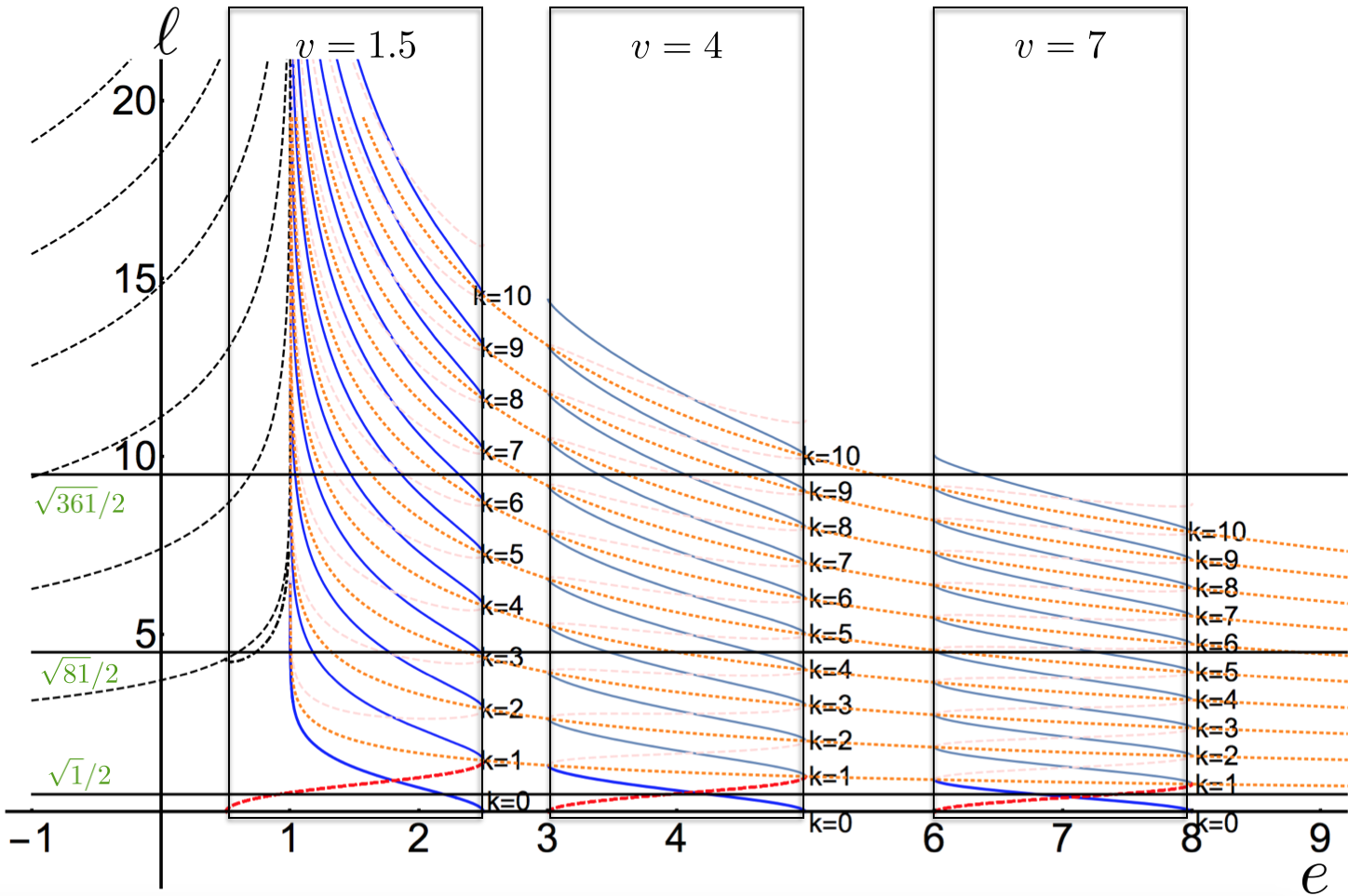}
   }
  \caption{For three different values of $v$, the figure presents the functions $\ell_a(e|v)$ (thick -- blue online), $\ell_b(e|v)$ (thick dashed -- red online), $2 \ell_b(e|0) - \ell_b(e|v)$ (thick dot-dashed -- black online) this function is only defined for $v<2$, $2 k \ell_b(e|0)$ (thin dashed -- black online), $\ell_c(e|k)$ (dotted grey -- orange online), $\ell_d(e|v,k)$ (thin grey -- blue online) and $\ell_e(e|v,k)$ (thin dashed light grey -- light red online). The intersection of each of these functions with the (thick black) horizontal at $\sqrt{M}/2$ determines the pseudo-energy of solutions of category $(a)$, $(b)$ single simple swing, $(b)$ single complex swing, $(b)$ multiple simple swings, $(c)$, $(d)$ and $(e)$ respectively. All discontinuous curves correspond to unstable solutions while the full curves correspond to stable solutions. The number of full loops $k$ reported to the right of each box referred to full and dotted lines of categories $(d)$ and $(c)$. The number of loops for the dashed lines corresponding to category  $(e)$ are $k-1$. 
   \label{fig-gather}}
\end{figure}

This simple system  proves to be quite rich. The study of Secs.~\ref{sec-class}~\&~\ref{sec-det} provides all its (non trivial) equilibria together with their stability.  For the illustrative purpose, we chose $M=81$ and $v=1.5$ for the different figures. In that case, the system has four stable and seven unstable equilibria. A direct computation of their internal energy shows that the case $k=2$ in category $(d)$ displayed in the middle column of Fig.~\ref{fig-d} is the global energy minimiser of the problem.  

The complete analysis is summarised in Fig.~\ref{fig-gather} where each curve correspond to one possible type of equilibrium. The full (resp. discontinuous) curves correspond to stable (resp. unstable) equilibria. For given $M$ and $v$ the pseudo-energies of all possible equilibria are indicated by the intersection of the curves in Fig.~\ref{fig-gather} with the horizontal at $\sqrt{M}/2$ the abscissa of which are in the interval $[v-1,v+1]$. From the figure we see that for  $v<2$ the function $\ell_d(e|k,v)$ has a vertical asymptote at $e=1$. Therefore, for increasing values of $M$ the thick horizontal grey (green online) line has more intersections with higher $k$ values (corresponding to more loops)  while  solutions of lower $k$ exist and are stable. When $v>2$ the number of stable solutions decreases as the asymptote can no longer be reached. For instance with $M=361$ there are 6 stable solutions when $v=1.5$ but only $3$ for $v=4$.

\section{Conclusion} 

In this paper, we obtained  geometric conditions for the positive definiteness of the second variation of a family of one-dimensional functionals.  A typical approach to prove stability for these problems is to consider the associated Sturm-Liouville problem and study numerically its spectrum. Such numerical studies can be delicate due to the sensitivity of the eigenvalues when a solution crosses a maximum of $V$ with $|\theta'|\ll1$. We presented a different approach by defining global indices based on the geometry of trajectories in phase plane. In many cases, these indices provide a complete solution to the stability problem.  Theorems~\ref{theo-IndI},~\ref{theo-IndJ}, and~\ref{theo-xi} constitute the main results. Taken together, they offer a powerful method to tackle many difficult stability issues without the need for numerical analysis, as shown in a physical example of a weighted hanging rod with intrinsic curvature.
We chose a simple but generic form for the functional as a starting point, but we expect that many of the arguments presented here could be generalised to other problems.\\

\noindent{\bf{Acknowledgments}}\\
We wish to thank John Maddocks and Apala Majumdar for fruitful discussions.

\appendix 

\section{Minimality with respect to perturbations in $\mathcal C^N[a,b])$ and minimality with respect to perturbation in $C^1([a,b])$\label{app-CnvsC1}} 

In Section~\ref{sec-state} we stated that to find a minimum of the functional\ret{myfunc}{FormL}, we could restrict the study to perturbations in $\mathcal C^N([a,b])$ instead of having to study the larger space of $C^1([a,b])$ perturbations.  In this paper we only consider the following two cases: either\re{secvar} holds or $\exists \tau \in \mathcal C^N([a,b]): \delta^2\mathscr E_\theta[\tau] <0$. Then the statement holds because of the following three propositions. 

\begin{proposition} 
For a given function $\theta\in C^1([a,b])$, the first variation of the functional\ret{myfunc}{FormL} vanishes $\forall \tau\in C^N([a,b])$ iff it vanishes $\forall \tau \in C^1([a,b])$. 
\end{proposition}
\begin{proof}
By direct computation of the first variation:
\ben{}
&&\forall \tau\in C^N([a,b]): \delta\mathscr E_\theta [\tau]=0.  \nonumber\\
&\Leftrightarrow &  \left . \pd{\mathscr L}{\theta'} \right|_{\theta(b),\theta'(b)} = \left . \pd{\mathscr L}{\theta'} \right|_{\theta(a),\theta'(a)} = 0, \quad \textrm{and} \qquad \left.\pd{\mathscr L}{\theta}\right|_{\theta(s),\theta'(s)} - \sd{}{s}\left . \pd{\mathscr L}{\theta '}\right|_{\theta(s),\theta'(s)}  = 0. \nonumber\\ 
&\Leftrightarrow&\forall \tau\in C^1([a,b]): \delta\mathscr E_\theta [\tau]=0.\nonumber
\een \qed
\end{proof}

Recall from section~\ref{sec-neu} that we can express the second variation of $\mathscr E$ as 
\be{}
\delta^2 \mathscr E_\theta[\tau] = K_{-V_{\theta\theta}\circ \theta,[a,b]} [\tau],
\ee
where for any function $f\in C^0([a,b])$, the functional $K_{f,[a,b]}$ was defined by\re{DefK} recalled here for convenience: 
\be{appdefK}
K_{f,[a,b]} [\tau] = \int_a^b {\tau'}^2(s)  + f(s)\,  \tau^2(s) ~~\textrm{d} s. 
\ee
We then have 
\begin{lemma}\label{appAlem} Let $f\in C^0([a,b])$  and the functional $K$ be defined according to\re{appdefK}, then 
\be{}
\forall \tau \in C^1([a,b]),~~\forall \eta>0,~~ \exists \bar \tau \in \mathcal C^N([a,b]): \left |K_{f,[a,b]}[\tau] - K_{f,[a,b]}[\bar \tau] \right |<\eta.
\ee
\end{lemma}
\begin{proof}
Let $\tau\in C^1([a,b])$, and $\epsilon\in \mathbb R_0^+ : ~ 0<\epsilon<b-a$. Also define the polynomial $p_\epsilon$ and $q_\epsilon$ as 
\ben{appdefp}
p_\epsilon(s)&=&  \tau (a+\epsilon) + \tau'(a+\epsilon) (s-(a+\epsilon))+ \frac{\tau'(a+\epsilon)}{2\epsilon} (s-(a+\epsilon))^2, \\
q_\epsilon(s)&=&  \tau (b-\epsilon) + \tau'(b-\epsilon) (s-(b-\epsilon))- \frac{\tau'(b-\epsilon)}{2\epsilon} (s-(b-\epsilon))^2, \label{appdefq}
\een
and remark that 
\ben{}
K_{f,[a,a+\epsilon]}[p_\epsilon] &=& \int_a^{a+\epsilon} {p_\epsilon'}^2 (s)  + f(s)~ p_\epsilon^2(s) \textrm{d}s \nonumber \\
&=&\epsilon  \int_{-1}^0 {\tau'}^2(a+\epsilon)  (1+z)^2 \nonumber\\
&&\qquad \qquad  + f(a+ \epsilon (z-1)) \left ( \tau(a+\epsilon)+ \epsilon \tau'(a+\epsilon) \left (z+\frac{z^2}{2}\right)\right)^2 ~~\textrm{d} z\nonumber\\
&=&O(\epsilon), \label{apppepsvanish}
\een
where the second equality comes after the change of variables $\epsilon z = s-(a+\epsilon)$.  A similar argument leads to 
\be{appqepsvanish}
K_{f,[b-\epsilon,b]}[q_\epsilon]= O(\epsilon). 
\ee

Then consider the following function
\be{}
\bar \tau(s) = 
\left \{ 
\begin{split}
p_\epsilon (s) & \qquad \textrm{if } s\in[a,a+\epsilon],\\
\tau(s) & \qquad \textrm{if } s\in(a+\epsilon,b-\epsilon),\\
q_\epsilon (s) & \qquad \textrm{if } s\in[b-\epsilon,b].
\end{split} 
\right .
\ee
Note that by construction, $\bar \tau\in\mathcal C^N([a,b])$. 

Finally, we have
\ben{}
K_{f,[a,b]}[\tau] &=& K_{f,[a,b]} [\bar \tau] - K_{f,[a,a+\epsilon]} [p_\epsilon]- K_{f,[b-\epsilon,b]} [q_\epsilon] +  K_{f,[a,a+\epsilon]} [\tau] +  K_{f,[b-\epsilon,b]} [\tau] \nonumber\\
&=& K_{f,[a,b]} [\bar \tau] + O(\epsilon),
\een
where the first equality comes from splitting the domains of integration and the second equality comes by substituting the third and fourth terms according to\ret{apppepsvanish}{appqepsvanish} and because the last two terms are integrals of bounded integrands over a domain of size $O(\epsilon)$. 
\qed\end{proof}

The two following propositions are then direct applications of Lemma~\ref{appAlem}.
\begin{proposition}\label{app-propinsta}
Let $\theta$ be a stationary function for the functional $\mathscr E$. Then the following holds: 
\be{}
\exists \tau \in C^1([a,b]): \delta^2 \mathscr E_\theta [\tau]<0 \quad
\Rightarrow \quad 
\exists \bar \tau \in C^N([a,b]): \delta^2 \mathscr E_\theta [\bar \tau]<0.
\ee
\end{proposition}

\begin{proposition}\label{app-propsta}
Let $\theta$ be a stationary function for the functional $\mathscr E$. Then the following holds: 
\ben{}
&&\exists \bar M>0:~~ \forall \bar \tau \in C^N([a,b]): \delta^2 \mathscr E_\theta [\bar\tau]\geq \bar M\int_a^b\bar \tau^2(s) \textrm{d} s \nonumber \\
&\Rightarrow &
\exists M>0:~~ \forall \tau \in C^1([a,b]): \delta^2 \mathscr E_\theta [\tau]\geq M\int_a^b \tau^2(s) \textrm{d} s. 
\een
\end{proposition}

Finally, note that the converse of Propositions~\ref{app-propinsta}~\&~\ref{app-propsta} are trivially true.

\section{Sufficient condition for local minimality\label{app-stronglypositive}}
Assume that $\theta$ is stationary for the functional\ret{myfunc}{FormL}: 
\be{appsp-firstvar}
\forall\tau \in\mathcal C^X([a,b]):~ \delta\mathscr E_\theta[\tau]=0;
\ee
 and that the second variation of $\mathscr E$ at $\theta$ is strongly positive with respect to the $\textrm{L}^2$ norm:
 \be{appsp-secvar}
\exists k\in\mathbb R_0^+:~\forall \tau \in \mathcal C^X([a,b]): \delta^2 \mathscr E_\theta [\tau] \geq k ~\int_a^b \tau^2(s)~\textrm{d}s. 
\ee 
We show that $\theta$ is locally minimal for $\mathscr E$. 
 
Recall from Section~\ref{sec-state} that $\theta$ is minimal if for all \textit{admissible perturbations} $\tau$, there exists a number $M>0$ such that for all $\epsilon\in [-M, M]\setminus\{0\}$,
 \be{appminbasics}
\mathscr E[\theta+ \epsilon \tau]> \mathscr E[\theta].
 \ee
 
 We compute
\ben{app-sp1}
\mathscr E[\theta+\epsilon \tau] &=& \int_a^b \frac 1 2 \Big( (\theta'-A) + \epsilon \tau'\Big)^2 - V(\theta+ \epsilon \tau) ~~\textrm{d} s, \\
&=&\int_a^b \frac {(\theta'-A)^2} 2  + \epsilon (\theta'-A) \tau' + \epsilon^2 \frac{{\tau'}^2}{2}\nonumber \\
	&&\qquad - \Big( V(\theta) + V_\theta(\theta) (\epsilon \tau) +  V_{\theta\theta}(\theta)  \frac {(\epsilon\tau)^2} {2} + \nu (\epsilon \tau;\, s) ~(\epsilon \tau)^2 \Big) ~~\textrm{d} s, \label{app-sp2}\nonumber \\
&=& \mathscr E[\theta] + \epsilon~ \delta \mathscr E_\theta[\tau]  + \frac{\epsilon^2} 2	~ \delta^2\mathscr  E_\theta[\tau]  - \frac{\epsilon^2}{2} \int_a^b 2 \nu (\epsilon \tau;\, s) ~\tau^2 ~\textrm{d} s  \label{app-sp3}\\
&\geq& \mathscr E[\theta] +\frac {\epsilon^2}{2} \left ( \big(k-\bar \nu(\epsilon)\big) \int_a^b \tau^2(s)~ \textrm{d} s    \right). \label{app-sp4}
\een
The second equality comes after rearranging the first term of the integrand and Taylor expanding the second term in\re{app-sp1}. The function $\nu(\epsilon \tau;\, s)$ is the prefactor of the remainder of this Taylor expansion. Accordingly, at each $s$: 
\be{appsp-nu1}
\nu(\epsilon \tau;\, s)\stackrel{\epsilon \tau(s) \to 0}{\to} 0.
\ee
The third equality comes after grouping the first and fourth, second and fifth and third and sixth terms in\re{app-sp3}. 
Finally, we define the function $\bar \nu(\epsilon)=\sup_{s\in[a,b]} |\nu (\epsilon \tau(s);\, s)|$ and note that\re{appsp-nu1} insures that 
\be{app-limnu}
\bar \nu(\epsilon)\stackrel{\epsilon\to 0}{\to} 0.
\ee
The inequality\re{app-sp4} is due to\re{appsp-firstvar} and\re{appsp-secvar}.

It is then always possible to choose $M$ such that the second term in the R.H.S. of\re{app-sp4} is strictly positive $\forall\epsilon\in[-M,M]\setminus\{0\}$.

\section{Sturm-Liouville problems\label{app-itworks}}

Let $p, q\in C^1([a,b])$ with $p>0$. We consider the following Sturm-Liouville problem with separate boundary conditions: 
\be{appSL}
(- p y')' + q \, y = \lambda y; \qquad y'(a)=0, \quad y'(\sigma)=0,
\ee
where $a< \sigma\leq b$. 

We first list a number of well known results regarding regular Sturm-Liouville problems with separate boundary conditions (see~\cite{kepe10} and reference therein): the eigenvalues $\lambda$ for which\re{appSL} admits a solution are separate, bounded from below and simple (the vectorial space of eigenfunctions associated to one eigenvalue is of dimension one). Furthermore eigenfunctions associated to different eigenvalues are orthogonal. 

Our results depend crucially upon the fact that eigenvalues $\lambda$ of the Sturm-Liouville problems\re{appSL} with separate boundary conditions are continuous functions of $\sigma$ (see~\cite{koze96} Theorem 3.1):
\begin{proposition}
\label{theo32}
If for $\sigma=\sigma_0$ there exists a solution to\re{appSL} with $\lambda = \lambda_0$, then for all $\epsilon>0$, there exists a $\delta$ such that if $|\sigma-\sigma_0|<\delta$, then there exits a solution of\re{appSL} with $|\lambda-\lambda_0|<\epsilon$.
\end{proposition}

It is however important to realise that, this theorem does not imply the continuity of the $k$-th eigenvalue. It only states that the existence of the eigenvalue $\lambda_0$ at $\sigma_0$ implies the existence and continuity of $\lambda$ as a function of $\sigma$ in some (arbitrarily small) open set around $\sigma_0$. It is in fact possible to build examples (see~\cite{koze96}) of Sturm-Liouville problems with slightly more complicated boundary conditions than that of\re{appSL} which  obey the assumptions of Proposition~\ref{theo32} but for which a new branch of eigenvalues appear at some value $\sigma=c$ with $a<c<b$ and $\lim_{\sigma\to c^+}=-\infty$. 

If such a branch of eigenvalues existed for the Sturm-Liouville operator $\mathcal S$ defined in Section~\ref{sec-num}, our argument would collapse. Indeed, following~\cite{ma09}, we proposed to count the number of negative eigenvalues of\re{appSL} with $\sigma=b$ by counting the number of inborn eigenvalues when $\sigma\to a^+$ and then keeping track of the change of signs of eigenvalues as $\sigma$ is continuously increased up to $b$. If negative eigenvalues can simply appear ``out of the blue'' without having to be positive eigenvalues that changed sign, the argument would fail. Let us first prove that 
\begin{proposition}
\label{noneginf}
If $p=1$ and there exists a number $Q$ such that $|q(x)|<Q$ $\forall x\in[a,\sigma]$, then all eigenvalues $\lambda$ of the problem\re{appSL} must respect $\lambda>-Q$.
\end{proposition}
\begin{proof}
Assume there exists a $\lambda\leq-Q$ such that\re{appSL} admits a solution $u$. By linearity of\re{appSL}, the function $h=u/u(a)$ is also a solution for the same $\lambda$. By construction, $h$ is the unique solution of the initial value problem
\be{appIVPpr1}
h''= (q-\lambda) h;\qquad h(a)=1,\quad h'(a)=0. 
\ee
But since we assumed $\lambda\leq-Q$, Eq.\re{appIVPpr1} implies $ h''\geq (q+Q) h(x)>0$ and $h'$ is a monotonous strictly  increasing function. It is therefore impossible that $h'(\sigma)=0$  and $h$ can not be a solution of\re{appSL}. A contradiction.\qed
\end{proof}

Since our Sturm-Liouville operator $\mathcal S$ respect the hypothesis of Proposition~\ref{noneginf}, there can not exist a branch of eigenvalues the limit of which is $-\infty$. Next, we must also rule out the possibility of new branches appearing on open sets with finite limits:
\begin{proposition}
\label{outoftheblue}
Let $\lambda(\sigma)$ be a continuous branch of eigenvalues of\re{appSL} with $\sigma\in(c_1,c_2)$ an open set with $a<c_1<c_2\leq b$ and with $p$ and $q$ respecting the assumptions of Proposition~\ref{noneginf}. If $\lambda_1= \lim_{\sigma\to c_1^+} \lambda(\sigma)<+\infty$, then  $\lambda_1$ is an eigenvalue of\re{appSL} with $\sigma=c_1$. Similarly, if $\lambda_2 = \lim_{\sigma\to c_2^-} \lambda(\sigma)<+\infty$, then  $\lambda_2$ is an eigenvalue of\re{appSL} with $\sigma=c_2$.
\end{proposition}
\begin{proof}
There are two cases to consider: either $\lim_{\sigma\to c_i}= \lambda_i<+\infty$ or $\lim_{\sigma\to c_i}=+\infty$. Since the property is not concerned with the latter, we focus on the former. Note that we can not have $\lim_{\sigma\to c_i}=-\infty$ because of Proposition~\ref{noneginf}. From now on, we therefore assume that $\lambda_i$ are finite. Consider the following family of initial value problems
\be{appIVP2}
-h''(x)+ (q(x)-\lambda(\sigma)) h=0;\qquad h(a,\sigma)=1,\quad h'(a,\sigma)=0.
\ee
After rescaling according to $x= a +(\sigma-a) z$, it is equivalent to the first order problem
\be{appIVP3}
\left\{ 
\begin{split}
\sd{v}{z} &= (\sigma-a) \bigg [q(a+ (\sigma-a) z) - \lambda(\sigma) \bigg] g,\\
\sd{g}{z} &= (\sigma-a) v,
\end{split}
\right .
\qquad \left\{ 
\begin{split}
g(0,\sigma)&=1\\
v(0,\sigma)&=0,
\end{split}
\right.
\ee
It is easy to see that if $\Big(g(z,\sigma),v(z,\sigma)\Big)$ solves\re{appIVP3} for a particular value of $\sigma$, then $h(x,\sigma)=g\left(\frac {x-a}{\sigma-a},\sigma\right)$ solves\re{appIVP2}. Furthermore $h'(x,\sigma)= \frac{1}{\sigma-a} \sd{g}{z}= v\left(\frac {x-a}{\sigma-a},\sigma\right)$. \textit{Vice versa} if $h(x,\sigma)$ solves\re{appIVP2} then $\Big( h[a+ (\sigma-a) z,\sigma],h'[a+ (\sigma-a) z,\sigma]\Big)$ solves\re{appIVP3}.

Since $\lambda(\sigma)$ is an eigenvalue of\re{appSL} for $\sigma\in(c_1,c_2)$,  we have $\left .\left(\pd{}{x}h(x,\sigma) \right)\right |_{x=\sigma}=0$. Hence if $(g,v)$ solves\re{appIVP3}, then $v(1,\sigma) =0~\forall\sigma\in (c_1,c_2)$. But since the solution $(g,v)$ of the IVP\re{appIVP3} is continuously dependent on both $x$ and the parameter $\sigma$ (see e.g.~\cite{ha80} Theorem 3.2), $\lim_{\sigma\to c_i} v(1,\sigma)=0$ which in turn implies that $\lim_{\sigma\to c_i} h'(\sigma)=0$ so that $h(x,c_i)$ is an eigenfunction of the Sturm-Liouville problem with eigenvalue $\lambda_i$.\qed 
\end{proof}

Together, Proposition~\ref{noneginf}~\&~\ref{outoftheblue} insure that in the case of the Sturm-Liouville operators\re{appSL} with Neumann boundary conditions, $p=1$ and bounded $q$ there can be no negative eigenvalues appearing ``out of the blue'' for $\sigma>a$. 

Finally, we must show that when there exists a $\sigma$ such that an eigenvalue vanishes: $\lambda(\sigma)=0$, then only one eigenvalue changes sign. Indeed, if two eigenvalues were to change signs at the same $\sigma$, then the counting argument exposed earlier would also fail. However, different continuous branches of eigenvalues never cross:
\begin{proposition}
Let $\lambda_1(\sigma)$ and $\lambda_2(\sigma)$ be continuous branches of eigenvalues of\re{appSL} such that there exists an open set $(a,\sigma_0)\subset (a,b]$ such that $\sigma \in(a,\sigma_0): \lambda_1(\sigma)\neq \lambda_2(\sigma)$. Then,~ $\lambda_1(\sigma_0)\neq\lambda_2(\sigma_0)$.
\end{proposition}
\begin{proof}
This is a consequence of~\cite{koze96} Theorem 3.2  which states that if $\lambda(\sigma_0)$ is an eigenvalue of\re{appSL}, and $u(.,\sigma_0)$ a normalised eigenfunction of $\lambda(\sigma_0)$, then there exist normalised eigenfunctions $u(.,\sigma)$ of $\lambda(\sigma)$ such that 
\be{}
u(.,\sigma)\to u(.,\sigma_0),\qquad \textrm{as }\sigma\to\sigma_0,
\ee
uniformly on any compact subinterval of $(a,\sigma_0]$.

As a result, if there existed a $\sigma_0$ such that $\lambda_1(\sigma_0) =\lambda_2(\sigma_0)=\lambda_0$, we could define an associated normalised eigenfunction $u_0$. But then by the theorem mentioned above there would also exist normalised eigenfunctions $u_1(.,\sigma)$ of $\lambda_1(\sigma)$ and $u_2(.,\sigma)$ of $\lambda_2(\sigma)$ such that $u_1\stackrel{\sigma\to\sigma_0}{\to}u_0$ and $u_2\stackrel{\sigma\to\sigma_0}{\to}u_0$. In particular, this would imply that $\lim_{\sigma\to\sigma_0} \int_a^\sigma ||u_1(x,\sigma)-u_2(x,\sigma)||^2\textrm{d}x = 0$. However for any $\sigma<\sigma_0$, the branches are different: $\lambda_1(\sigma)\neq\lambda_2(\sigma)$ and the associated eigenfunctions must be orthogonal. Therefore $\int_a^\sigma ||u_1(x,\sigma)-u_2(x,\sigma)||^2\textrm{d}x =2$ and the integral can not vanish in the limit $\sigma\to\sigma_0$.\qed
\end{proof}

\end{document}